\documentclass[10pt, a4paper]{amsart}

\usepackage[utf8]{inputenc}
\usepackage[english]{babel}
\usepackage{url}
\usepackage{xcolor}
\usepackage[all,cmtip]{xy}
\usepackage{graphicx}
\usepackage{adjustbox}

\usepackage{hyperref}

\usepackage{amsmath,amssymb,enumerate,mathrsfs,mathtools,amsthm,emptypage,array,xpatch,faktor,xfrac}
\usepackage{tikz}

\newtheoremstyle{thmstyle}
  {}
  {}
  {\itshape}
  {}
  {\bfseries}
  {.}
  { }
  {}
\theoremstyle{thmstyle}
\newtheorem{thm}{Theorem}[section]
\newtheorem{corollary}[thm]{Corollary}
\newtheorem{proposition}[thm]{Proposition}
\newtheorem{lemma}[thm]{Lemma}

\newtheorem*{que*}{Question}
\newtheorem*{thm*}{Theorem}

\newtheorem{step}[thm]{Step}

\newtheoremstyle{defstyle}
  {}
  {}
  {}
  {}
  {\bfseries}
  {.}
  { }
  {}
\theoremstyle{defstyle}
\newtheorem{definition}[thm]{Definition}
\newtheorem{remark}[thm]{Remark}
\newtheorem*{remark*}{Remark}
\newtheorem{remdef}[thm]{Remark - Definition}

\newtheorem{example}[thm]{Example}
\newtheorem{setup}[thm]{Setup}
\newtheorem*{generality}{Generality conditions}

\newcommand{\Proj}{\operatorname{Proj}}
\newcommand{\Spec}{\operatorname{Spec}}

\newcommand{\Pic}{\operatorname{Pic}}

\newcommand{\Ne}{\operatorname{NE}}
\newcommand{\NE}{\overline\Ne}
\newcommand{\Exc}{\operatorname{Exc}}
\newcommand{\codim}{\operatorname{codim}}
\newcommand{\eff}{\operatorname{Eff}}
\newcommand{\Eff}{\overline\eff}
\newcommand{\Nef}{\operatorname{Nef}}
\newcommand{\ph}{\varphi}
\newcommand{\w}{\widetilde}
\newcommand{\ma}{\mathcal}

\newcommand{\ol}{\mathcal{O}}
\newcommand{\wi}{\widehat}
\newcommand{\pr}{\mathbb{P}}
\newcommand{\PP}{{\pr^1\times\pr^1}}
\newcommand{\af}{\mathbb{A}}
\newcommand{\Q}{\mathbb{Q}}
\newcommand{\C}{\mathbb{C}}
\newcommand{\R}{\mathbb{R}}
\newcommand{\Z}{\mathbb{Z}}

\newcommand{\Bl}{\operatorname{Bl}}

\newcommand{\Bs}{\operatorname{Bs}}
\newcommand{\virg}[1]{\textquotedblleft #1\textquotedblright}

\title{The blow-up of the cubic fourfold along eight planes is Fano}
\author{Saverio Andrea Secci}
\address{Saverio Andrea Secci, SISSA -- Scuola Internazionale Superiore di Studi Avanzati, Via Bonomea 265, 34136 Trieste, Italy}
\email{ssecci@sissa.it}

\makeatletter
\@namedef{subjclassname@2020}{\textup{2020} Mathematics Subject Classification}

\makeatother

\calclayout 

\numberwithin{figure}{section}
\numberwithin{table}{section}
\numberwithin{equation}{section}

\begin{document}
\maketitle
\begin{abstract}
We prove that the blow-up $Z_m$ of a smooth cubic fourfold along $m\leq8$ planes, pairwise meeting at distinct points and satisfying explicit generality conditions, is a smooth Fano fourfold of Picard rank $\rho(Z_m) =m+1\leq9$; in particular, we obtain three new families of smooth Fano fourfolds with $\rho\geq7$ that are not a product of surfaces.
Complementary to this result, we show that a smooth cubic fourfold as above exists.

Moreover, we study the birational geometry of $Z_m$ and show that it corresponds exactly to that of a smooth del Pezzo surfaces $S_m$, obtained as the blow-up of $\pr^2$ along $m\leq8$ points in general position.
Namely, $Z_m$ and $S_m$ have isomorphic Mori cones, pseudo-effective cones with same Mori chamber decomposition, and there is a one-to-one correspondence between the contractions of $Z_m$ and the contractions of $S_m$.

Lastly, we study the anticanonical linear system of $Z_m$ and prove that it is basepoint free if and only if $m\leq7$.
For $m=8$ and a special configurations of planes, we show that $\Bs(|-K_{Z_8}|)$ has dimension $2$ and contains a smooth del Pezzo surface of degree $1$.

\end{abstract}
\section{Introduction}
The classification of smooth complex Fano varieties has been a central problem in both classical and modern algebraic geometry: from the classification of del Pezzo surfaces by the Italian school in the 1920s, to the classification of Fano threefolds in the 1980s by Mori and Mukai for Picard rank $\rho \geq 2$ and by Iskovskikh and Prokhorov for Picard rank $\rho=1$, to the boundedness in arbitrary dimension by Koll{\'a}r, Miyaoka and Mori \cite{KMMo2} in the early 1990s.
It then became clear that the classification of higher-dimensional smooth Fano varieties was theoretically possible and this fuelled modern research on the subject.

Already in dimension $4$, only the classification above a certain bound is complete: in a striking recent result \cite{Casa25}, building on \cite{Casa24}, Casagrande proved that smooth Fano fourfolds with Picad rank $\rho > 9$ are products of del Pezzo surfaces, generalizing Mori and Mukai's analogous result in dimension $3$ \cite{MM81}.
Below this threshold, however, the situation is considerably richer, and only partial classification results are currently available.
At present, only $5$ families of smooth Fano fourfolds with $\rho \geq 7$ that are not a product of surfaces are known: the Fano model of the blow-up of $\pr^4$ along $m\in\{6,7,8\}$ points in general position \cite[Example 7.2 and references therein]{CS24}, with $\rho=m+1\in\{7,8,9\}$, and two new examples with $\rho=7$ constructed in \cite[Proposition 1.9]{CS24}.

One of the cases studied in \cite{Casa25} is the class of Fano fourfolds that are not a product of surfaces, and with no small contractions.
The following is a structure theorem.
\begin{thm*}{\cite[Theorem 1.2]{Casa25}}
Let $X$ be a smooth Fano fourfold with $\rho(X)\geq 7$, not isomorphic
to a product of surfaces, and with no small elementary contractions.
Then, $\rho(X)\leq 9$, and there are a cubic fourfold $Z\subset\pr^5$ with at most isolated ordinary double points, and $m := \rho(X)-1$ planes $\Pi_1,\dots,\Pi_m\subset Z$ pairwise meeting at distinct points, such that $X$ is obtained from $Z$ by blowing up one plane and successively the transforms of the other ones.
\end{thm*}

Reversing the process when $Z$ is smooth, that is showing that such a blow-up is indeed Fano, was previously known only for $m=1,2$ \cite[$\S11.7$]{Casa25}, while the remaining cases $3\leq m \leq 8$ were open; see \cite[Question 11.6]{Casa25}.
In this paper we answer this question for a general smooth cubic fourfold by proving the following
\begin{thm}\label{main:fano}
Let $Z\subset \pr^5$ be a general smooth cubic fourfold containing $m \leq 8$ planes $\Pi_1,\dots,\Pi_m$ pairwise meeting at distinct points\footnote{By \emph{general} we mean belonging to a non-empty open subset, not necessarily dense, of the corresponding parameter space within $|\ol_{\pr^5}(3)|$, which may be reducible.}.
Then, the blow-up $\sigma \colon Z_m \to Z$ along $\Pi_1\cup\dots\cup \Pi_m$ is a smooth Fano fourfold with no small contractions.
\end{thm}
Note that, since $Z$ is smooth, blowing up one plane and successively the transforms of the other ones or blowing up the scheme-theoretic union of the planes gives isomorphic varieties; see Corollary \ref{blowup-planes}.

One can see that $Z_m$ is not a product of surfaces; hence $\rho(Z_m)\leq 9$ by \cite[Theorem 1.1]{Casa25}, and the bound $m=\rho(Z_m)-1\leq 8$ is sharp.
In particular, with Theorem \ref{main:fano} we construct three new families of smooth Fano fourfolds with Picard rank $\rho\geq 7$, one for each value of $\rho\in\{7,8,9\}$, bringing the total number of known families to $8$; we refer to Table \ref{final table} for their numerical invariants.

Moreover, the generality assumption in Theorem \ref{main:fano} can be made precise: we identify explicit conditions -- the generality conditions $\boldsymbol{(\star)}$ -- which are both necessary and sufficient for $Z_m$ to be Fano; see $\S$\ref{gen_cond} and Theorem \ref{m=8}.

\medskip

As part of the proof of Theorem \ref{main:fano} we study the birational geometry of $Z_m$, which we fully describe.
\begin{thm}\label{main:cone}
Let $Z_m$ be as in Theorem \ref{main:fano}, and let $S_m$ be a smooth del Pezzo surface obtained as the blow-up $g\colon S_m\to \pr^2$ at $m\leq 8$ points in general position.
Then:
\begin{enumerate}[i.]\setlength\itemsep{4pt}
\item $\NE(Z_m)\cong\NE(S_m)$, and there is a one-to-one correspondence between the contractions of $Z_m$ and the contractions of $S_m$;
\item $\Eff(Z_m)\cong\Eff(S_m)$, with isomorphic Mori chamber decompositions.
\end{enumerate}
\end{thm}
Recall that smooth Fano varieties are Mori dream spaces \cite[Corollary 1.3.2]{BCHM}, and we refer to \cite{HK00} for the Mori chamber decomposition.

An important aspect of Theorem \ref{main:cone} is the remarkable correspondence between $Z_m$ and $S_m$, where the blow-up $\sigma \colon Z_m \to Z$ along the union of $m$ planes mirrors exactly the blow-up $g\colon S_m\to \pr^2$ at $m$ points in general position.
We prove that $Z_m$ and $S_m$ share the same birational geometry, which is governed by the combinatorial data of the Mori cone (item \emph{i.}) and the pseudo-effective cone (item \emph{ii.}); see Corollaries \ref{Eff(X)}, \ref{Nef(X)} and \ref{contractions:X} for more detailed statements.

These isomorphisms of cones are induced by linear isomorphisms $\ma N_1(Z_m)\to\ma N_1(S_m)$ and $\ma N^1(Z_m)\to\ma N^1(S_m)$, respectively; see Remark \ref{alpha}.
We have $\NE(S_m)=\Eff(S_m)$ since $S_m$ is a surface, hence $\NE(Z_m)\cong\Eff(Z_m)$ by Theorem \ref{main:cone}: to each extremal ray $R$ of $\NE(Z_m)$ whose elementary contraction $c_R$ is divisorial, there corresponds the extremal ray of $\Eff(Z_m)$ generated by the class of $\Exc(c_R)$.

It is natural to ask whether $S_m$ can be identified with some special subvariety of $Z_m$, but this is not always the case; see Remark \ref{embedded surface}.

\medskip

Finally, we study the anticanonical linear system $|-K_{Z_m}|$ and prove the following.
\begin{thm}\label{thm:base-locus}
Let $Z_m$ be as in Theorem \ref{main:fano}.
Then, $-K_{Z_m}$ is globally generated if and only if $m\leq7$.

If $m=8$ and $\Pi_1,\dots,\Pi_8$ are tangent to a Veronese surface $\ma V\subset \pr^5$, then the base locus $\Bs(|-K_{Z_8}|)$ has dimension $2$ and contains a smooth del Pezzo surface of degree $1$.
\end{thm}
The special configuration of planes tangent to a Veronese surface $\ma V\subset \pr^5$ plays an important role throughout the entirety of the paper, as it presents us with an explicit example having unique symmetries.

This special configuration yields in fact two results.
We show that a smooth cubic fourfold $Z$ containing general planes $\Pi_1,\dots,\Pi_8$ tangent to $\ma V$ always contains a ninth plane $\Pi_9$ tangent to $\ma V$, whose strict transform in $Z_8$ is an irreducible component of $\Bs(|-K_{Z_8}|)$; see Corollary \ref{co:base-locusV} and Remark \ref{last-added remark}.
Moreover, we see in Corollary \ref{abelian fibration} that if the scheme-theoretic base locus $\Bs(|-K_{Z_8}|)$ is smooth and irreducible, then the blow-up $Z_9 \to Z$ along $\Pi_1\cup\dots\cup\Pi_9$ -- equivalently, the blow-up $Z_9 \to Z_8$ along $\Bs(|-K_{Z_8}|)$ -- admits an abelian fibration $\Phi\colon Z_9 \to \pr^2$ such that $-K_{Z_9}\simeq\Phi^*\ol_{\pr^2}(1)$.
This further strenghtens the correspondence between $Z_8$ and $S_8$ in this special case, where we identify $\Pi_9$ with the ninth point in the base locus of plane cubics through $8$ points in general position -- similarly, $\Bs(|-K_{Z_8}|)$ with the point $q=\Bs(|-K_{S_8}|)$ -- and $\Phi\colon Z_9 \to \pr^2$ with the elliptic fibration $\Bl_{q} S_8\to \pr^1$, making $\sigma \colon Z_8\to Z$ a full-fledged \textbf{four-dimensional counterpart} of smooth del Pezzo surfaces of degree $1$.

\subsection{General strategy of the proof and structure of the paper}
The paper is structured as follows.
In Section \ref{pre} we introduce the notation and prove some preliminary results, while in Section \ref{ssetup} we describe the construction of $Z_m$, see Setup \ref{setup},  and recall its geometry in the known cases $m\in\{1,2\}$, see $\S$\ref{z1} and $\S$\ref{z2}.
Moreover, in $\S$\ref{gen_cond} we precisely identify the necessary conditions -- the generality conditions $\boldsymbol{(\star)}$ -- that a smooth cubic fourfold has to satisfy in order for $Z_m$ to be Fano.

Section \ref{fano} is devoted to the proof of Theorems \ref{main:fano} and \ref{main:cone}. First, we show in Proposition \ref{m=3} that $Z_3$ is Fano, and determine equivalent conditions to the generality conditions $\boldsymbol{(\star)}$ in Corollary \ref{co:line-types}.
Then, we prove the following result; it showcases the first correspondence between $Z_3$ and the blow-up $g\colon S_3\to \pr^2$ at $3$ non-collinear points.
\begin{thm}\emph{[Theorem \ref{thm:cremona} and Remark \ref{re:cremona}]}\label{main:cremona}
Let $Z\subset \pr^5$ be a general cubic fourfold containing planes $\Pi_1,\Pi_2,\Pi_3$ pairwise meeting at distinct points. The linear system
$$|H^0(\ma I_{\Pi_1\cup\Pi_2\cup\Pi_3/Z}\otimes\ol_Z(2))|$$
defines a rational map $\zeta\colon Z \dashrightarrow Z'$ onto a smooth cubic fourfold $Z'\subset\pr^5$ whose resolution $\sigma\colon Z_3 \to Z$ is such that $\sigma':=\zeta\circ\sigma\colon Z_3 \to Z'$ is the blow-up along planes $\Pi_1',\,\Pi_2',\,\Pi_3'\subset Z'$ pairwise meeting at distinct points.
\end{thm}
The rational map $\zeta\colon Z \dashrightarrow Z'$, which we call a \emph{Cremona-like transformation}, mirrors the standard Cremona transformation of $\pr^2$ centred at the coordinate points.
In other words, the exceptional divisors of $\sigma'\colon Z_3 \to Z'$ are the strict transform of the unique hyperplane sections of $Z$ spanned by $\Pi_i\cup\Pi_j$ for $i\neq j$, similarly to the other blow-up $g'\colon S_3 \to \pr^2$ of $S_3$, whose exceptional curves are the strict transform of the unique lines of $\pr^2$ spanned by $2$ of the $3$ blown-up points of $g\colon S_3\to \pr^2$.

Theorem \ref{main:cremona} allows us to prove Theorem \ref{m=8}, the final step towards Theorems \ref{main:fano} and \ref{main:cone}.
We fix an isomorphism between the Picard groups of $S_m$ and $Z_m$ (see Remark \ref{alpha}), and we use this identification to extend the action of the group of Cremona isometries of $S_m$ on $\Pic(S_m)$ (cf.\ \cite{Dolg83, Dolg12}) to $\Pic(Z_m)$.
Doing so, we are able to fully describe the Mori cone $\NE(Z_m)$ and to show that $Z_m$ is Fano for all $m \leq 8$.

However, Section \ref{fano} leaves us with an open question: does a smooth cubic fourfold, containing $m\leq 8$ planes as in Setup \ref{setup} and satisfying the generality conditions $\boldsymbol{(\star)}$, exist? We provide the answer in Section \ref{sec:existence} and prove the following.
\begin{thm}\label{existence}
There exists a smooth cubic fourfold $Z$ containing $m\leq 8$ planes $\Pi_1,\dots,\Pi_m$ as in Setup \ref{setup} and satisfying the generality conditions $\boldsymbol{(\star)}$.
\end{thm}
An essential tool for the proof of Theorem \ref{existence} is the above-mentioned special configuration of planes tangent to a Veronese surface (see $\S$\ref{veronese}), as well as the Horace method (cf.\ \cite{Horace}), which we use to compute the dimension of the space of global sections of cubic fourfolds containing $m\leq 8$ planes as in Setup \ref{setup}; see Proposition \ref{expected-dim} and Remark \ref{anticanonical sections}.

Finally, in Section \ref{abl} we prove Theorem \ref{thm:base-locus}, while in Section \ref{final} we compute the numerical invariants of $Z_m$, collect them in Table \ref{final table}, and give a comment on the rationality of $Z_m$.

\subsection{Related results and analogous constructions}
In the framework of smooth Fano fourfolds $X$ with large anticanonical base locus, first studied in \cite{HS25,HS25bis}, Theorem \ref{thm:base-locus} shows that $Z_8$ is a smooth Fano fourfold whose anticanonical base locus contains a smooth irreducible surface, provided that the blown-up planes $\Pi_1,\dots,\Pi_8$ are tangent to a Veronese surface: this gives a new example (cf.\ \cite[Problem 1.7]{HS25}).
More precisely, $Z_8$ complements the classification in \cite[Theorem 1.2]{HS25bis}, which assumes $h^0(X,-K_X)\geq 4$, since $h^0(Z_8,-K_{Z_8})=3$ (see Table \ref{final table}).
To the best of our knowledge, $Z_8$ is the first known smooth Fano fourfold with $h^0(X,-K_X)=3$ and a two-dimensional anticanonical base locus.
Moreover, if $\Bs(|-K_{Z_8}|)$ is irreducible, then $Z_8$ satisfies the hypotheses of \cite[Theorem 1.3]{HS25}, hence a general anticanonical divisor of $Z_8$ is not $\Q$-factorial and, in particular, it is singular.

\medskip

Finally, note that families of smooth Fano fourfolds with an analogous construction to the one of Theorem \ref{main:fano} also appear in \cite{Manivel24}, \cite{HS25bis}, and \cite{FanoQuiver26}.

In \cite[Theorem]{Manivel24} Manivel proves that the blow-up of a smooth quadric fourfold along $5$ general planes meeting pairwise at a point is Fano, and it follows that same holds for all $m\leq 5$.
Whether such a blow-up remains Fano for $m\in\{6,7,8\}$ is still an open problem, see \cite[$\S11.6$ and references therein]{Casa25}.

Another example is the Fano fourfold from \cite[Example 5.17]{HS25bis}, which is obtained as the blow-up of a smooth del Pezzo fourfold $X$ of degree $1$, i.e.\ $-K_X\simeq 3A$ and $A^3=1$, along the union of two del Pezzo surfaces of degree $1$ meeting at the basepoint $\Bs(|A|)=\{pt\}$.

Lastly, two more Fano fourfolds that are the blow-up along smooth surfaces pairwise meeting at a point appear in \cite{FanoQuiver26}: they are obtained as moduli spaces of representations of subspace quivers.
The Fano fourfold M$(1^5,2;3)$ coincides with the one from \cite{Manivel24}, see \cite[Theorem 34]{FanoQuiver26}, while the Fano fourfold M$(1^4,2^2;3)$ can be constructed as the blow-up of $\pr^2\times \pr^2$ along four surfaces of type $\{line\}\times \{line\}$, see \cite[Theorem 46]{FanoQuiver26}.
To conclude, note that M$(1^4,2^2;3)$ has the same numerical invariants of the Fano fourfold with $\rho=6$ from \cite[Proposition 7.24]{CS24}: the constructions in \cite{FanoQuiver26} and \cite{CS24} are different, but the two varieties are most likely isomorphic.

\bigskip

\noindent{\bf Acknowledgements.}
I thank Cinzia Casagrande for the many fruitful conversations and for a careful reading of a preliminary version of this paper, and I thank Andreas H\"oring for the many valuable inputs he gave me.

I am also grateful to Stefano Canino for introducing me to the Horace method and  to the properties of the tangent planes to the Veronese surface.

Moreover, I thank Massimiliano Mella for a valuable discussion on Cremona equivalence.

The author is a member of GNSAGA, INdAM.

\section{Preliminaries}\label{pre}
\subsection{Notations}
We refer to \cite{Hart} for general notions on schemes; we work over $\mathbb C$.
Schemes are assumed to be projective, and varieties are integral schemes. 

We refer to \cite{LazI, LazII} for the notions of positivity of divisors and vector bundles, to \cite{Debarre, KM} for the terminology and properties in birational geometry, and to \cite{HK00} for Mori dream spaces and  the Mori chamber decomposition.
Smooth Fano varieties are Mori dream spaces by \cite[Corollary 1.3.2]{BCHM}.

Given linear subspaces $\Lambda_1,\Lambda_2\subset \pr^n$, we donete by $\Lambda_1+\Lambda_2$ the linear subspace generated by $\Lambda_1\cup\Lambda_2$.
Given distinct points $p,q\in \pr^n$, we denote by $\overline{p\, q}$ the line $p+q$.

We denote by $\ell(X)$ the length of a zero-dimensional scheme $X$.

We denote an isomorphism of objects by $\cong$, linear equivalence of divisors by $\simeq$, and numerical equivalence (both of curves and divisors) by $\equiv$.

\medskip

For a normal variety $X$, we denote by $\rho(X)$ the Picard rank of $X$; $\ma N_1(X)$ and $\ma N^1(X)$ are the real vector spaces of classes of $1$-cycles with real coefficients and of $\R$-divisors, respectively, modulo numerical equivalence; $\dim\ma N_1(X)=\dim\ma N^1(X)=\rho(X)$.
We denote classes of $1$-cycles by [-].

A contraction $\ph\colon X \to Y$ is a surjective morphism with connected fibres between normal projective varieties.
We say that $\ph$ is elementary if $\rho(X)-\rho(Y)=1$.

A supporting divisor $D$ of a contraction $\ph\colon X \to Y$ is the pullback $D\simeq\ph^*A$ of some ample Cartier divisor $A$ on $Y$.

If $D$ is a semiample divisor on $X$, then the Stein factorization of the morphism corresponding to the linear system $|nD|$ for $n\gg0$ is a contraction with supporting divisor $D$.

A conic bundle is a contraction such that every fibre is a plane conic; a quadric bundle is a contraction such that every fibre is a quadric surface.

\medskip

Given a convex, pointed cone $\ma C$, an extremal face $F$ of $\ma C$ is a convex subcone of $\ma C$ that satisfies the following: for every $v,w \in\ma C$ such that  $v+w \in F$, then $v,w \in F$.
An extremal ray of $\ma C$ is a one-dimensional extremal face; a facet of $\ma C$ is an extremal face of codimension one.

\medskip

Let now $X$ be a (normal, $\Q$-factorial) Mori dream space: the Mori cone $\NE(X)$, the nef cone $\Nef(X)$ and the pseudo-effective cone $\Eff(X)$ are closed, rational and polyhedral.
We have $\Nef(X)\subseteq\Eff(X)$. The cones $\NE(X)$ and $\Nef(X)$ are dual to each other via the intersection form.
If $X$ is a surface, then $\ma N^1(X)=\ma N_1(X)$ and $\Eff(X)=\NE(X)$.

Given a contraction $\ph\colon X \to Y$, we denote by $\NE(\ph)$ the extremal face of $\NE(X)$ generated by the curves contracted by $\ph$.
If $D$ is a supporting divisor of $\ph$, we have $\NE(\ph)=\NE(X)\cap D^{\perp}$, where $D^{\perp}\subset\ma N_1(X)$ is the hyperplane of $1$-cycles $\Gamma$ such that $D\cdot\Gamma=0$.

We have $\dim\NE(\ph)=\rho(X)-\rho(Y)$: $\ph$ is elementary if and only if $\NE(\ph)$ is an extremal ray.

\subsection{Local computations}
\begin{lemma}\label{local blowup (n-2)-planes}
Let $\Pi_1,\, \Pi_2\subset \af^n$ be general $(n-2)$-planes. The blow-up $X \xrightarrow{\sigma} \af^n$ along $\Pi_1\cup \Pi_2$ is isomorphic to
$$X_2 \xrightarrow{\sigma_2} X_1 \xrightarrow{\sigma_1} \af^n$$
where $X_1 \xrightarrow{\sigma_1} \af^n$ is the blow-up along $\Pi_1$ and $X_2 \xrightarrow{\sigma_2} X_1$ is the blow-up along the strict transform of $\Pi_2$.
\end{lemma}
\begin{proof}
We have $\Pi_1\cap\Pi_2=\emptyset$ when $n=2,3$ and the statement is clear. Otherwise, let us assume $n\geq4$ and let $\bar x:=(x_1,\dots,x_n)$ be the coordinates of $\mathbb A^n$. We can assume that $\Pi_1=V(x_1,\, x_2)$ and $\Pi_2=V(x_3,\, x_4)$, so that the ideal of $\Pi_1\cup \Pi_2$ is $I:=(x_1x_3,\, x_1x_4,\, x_2x_3,\, x_2x_4)$. If we denote the ring $\C[x_1,\dots,x_n]$ by $R$, then $X:=\Proj (A)$ is given by the Rees algebra
$$A:=R \oplus I \oplus I^2 \oplus \dots\ .$$
We construct an embedding $X_2 \subset \af^n\times\pr^3$ so that $X_2=\Proj (R[a_0, \dots, a_3]\slash J)$ for some ideal $J$, and show that $A\cong R[a_0, \dots, a_3]\slash J$.

The ideal of $\Pi_1$ in $\af^n$ is generated by the regular sequence $\{x_1,x_2\}\subset R$, so $X_1\subset \af^n\times\pr^1$ with coordinates $\bar x\times(s:t)$ is given by $X_1=V(tx_1-sx_2)$. The ideal of the strict transform of $\Pi_2$ in $X_1$ is generated by the regular sequence $\{x_3,x_4\}\subset R[s]\slash (x_1-sx_2)$ (resp.\ $R[t]\slash (tx_1-x_2)$) in the affine chart $\{t\neq0\}\subset \af^n\times \af^1$ (resp.\ in $\{s\neq0\}$). Thus, $X_2\subset \af^n\times\pr^1\times\pr^1$ with coordinates $\bar x\times(s:t)\times(\alpha:\beta)$ is given by $$X_2=V(tx_1-sx_2, \, \beta x_3-\alpha x_4).$$
Through the Segre embedding to $\pr^3\ni(a_0:a_1:a_2:a_3)$
$$\pr^1\times\pr^1 \hookrightarrow \pr^3, \quad (s:t)\times(\alpha:\beta) \mapsto(s\alpha:s\beta:t\alpha:t\beta) \ ,$$
we see that $X_2 \subset \af^n\times \pr^3$ is given by the ideal $$J=(a_0a_3-a_1a_2,\ a_2x_1-a_0x_2,\ a_3x_1-a_1x_2,\ a_1x_3-a_0x_4,\ a_3x_3-a_2x_4).$$
Finally, to show that $X\cong X_2$, notice that $J$ is the kernel of the surjective morphism
$$R[a_0, \dots, a_3] \to A, \quad a_0\mapsto x_1x_3, \ a_1\mapsto x_1x_4,\ a_2\mapsto x_2x_3,\ a_3\mapsto x_2x_4.$$
\end{proof}

\begin{corollary}\label{blowup-planes}
Let $Z$ be a smooth projective fourfold, and let $Y_1,\, Y_2\subset Z$ be smooth surfaces meeting transversally at a point. Denote by $X_1\xrightarrow{\pi_1}Z$ and $X_2\xrightarrow{\pi_2}Z$ the blow-ups of $Z$ along $Y_1$ and $Y_2$, respectively.
Then,
$$X:=\Bl_{Y_1\cup Y_2}Z \cong \Bl_{\pi_1^{-1}(Y_2)}X_1 \cong\Bl_{\pi_2^{-1}(Y_1)}X_2.$$
In particular, $X$ is smooth.
\end{corollary}
\begin{proof}
It follows directly from Lemma \ref{local blowup (n-2)-planes}, since two general planes in $\af^4$ meet along a point.
See also \cite[Lemma 2.10]{Casa25} for another proof.
\end{proof}

\subsection{Exact sequences of ideal sheaves}\label{ses}
Let $X,Y\subset W$ be projective schemes. We recall two classical short exact sequences of ideal sheaves, that is
\begin{equation}\label{union}
 0\to\ma I_{X\cap Y/Y}\to \ol_{X\cup Y}\to \ol_X\to0
\end{equation}
and
\begin{equation}\label{union2}
 0\to\ma I_{X\cup Y/W}\to \ma I_{X/W}\to\ma I_{X\cap Y/Y} \to0.
\end{equation}
If we locally set $W=\Spec(R)$, $X=\Spec(R/I_X)$ and $Y=\Spec(R/I_Y)$, then \eqref{union} and \eqref{union2} correspond to
\begin{equation*}
 0\to(I_X+I_Y)\cdot R/I_Y\to R/(I_X\cap I_Y)\to R/I_X\to0
\end{equation*}
and
\begin{equation*}
 0\to I_X\cap I_Y \to I_X\to (I_X+I_Y)\cdot R/I_Y \to0.
\end{equation*}

Moreover, suppose that $W$ is integral and let $D$ be a prime divisor on $W$. The following exact sequence, a variation of \eqref{union2}, is known as the \emph{residual exact sequence}
\begin{equation}\label{horace}
 0\to\ma I_{{\rm Res}_D(X)/W}\otimes\ol_W(-D)\to \ma I_{X/W}\to \ma I_{X\cap D/D}\to0,
\end{equation}
where ${\rm Res}_D(X)\subset W$ is the \emph{residual} of $X$ with respect to $D$, that is the scheme defined by the conductor ideal sheaf $\ma I_{{\rm Res}_D(X)/W}:=(\ma I_{X/W}\colon \ma I_{D/W})$.
The residual first appeared in \cite[$\S21$]{Castelnuovo} and was then formalised in \cite[$\S4.1$]{Horace}, where Hirschowitz introduced the so-called \emph{Horace method}, an inductive method which uses \eqref{horace}.
Locally, \eqref{horace} corresponds to
\begin{equation*}
 0\to\bigl(I_X\colon(f)\bigr)\cdot (f)\to I_X\to \bigl(I_X+(f)\bigr)\cdot R/(f)\to0,
\end{equation*}
where we set $D=\Spec(R/(f))$.
If $X$ is reduced, then ${\rm Res}_D(X)$ is simply the Zariski closure of $X \setminus (X \cap D)$; if $X$ is zero-dimensional, then $\ell(X)=\ell(X\cap D) +\ell({\rm Res}_D(X))$.

\subsection{Auxiliary results}
\begin{lemma}\label{ampleness}
Let $X$ be a smooth projective variety. If $-K_X$ is semiample and strictly nef, then $X$ is Fano.
\end{lemma}
\begin{proof}
For $n\gg0$, $-nK_X$ is linearly equivalent to the pullback of an ample Cartier divisor via the finite morphism induced by $|-nK_X|$. Thus, $-K_X$ is ample.
\end{proof}

The following is a standard computation; we provide a proof for the reader's convenience.
\begin{lemma}\label{conti}
Let $\sigma\colon X \to Z$ be the blow-up of a smooth projective fourfold $Z$ along a smooth surface $S$, and let $E:=\Exc(\sigma)$. Then, 
$$(\sigma^*D-E)^4=D^4-6(D_{|S})^2-4D_{|S}\cdot c_1(\ma N_{S/Z}^*) +c_2(\ma N_{S/Z}^*)-c_1(\ma N_{S/Z}^*)^2$$
for all $D\in \Pic (Z)$.
\end{lemma}
\begin{proof}
We have $(\sigma^*D)^4=D^4$, so it suffices to compute $(\sigma^*D)^i\cdot E^{4-i}$ for $i\in\{0,\dots, 3\}$. Note that $\xi_E:=-E_{|E}$ is the tautological class of $E\cong\pr_S(\ma N_{S/Z}^*)$, so by \cite[Appendix A]{Hart} we have $$\xi_E^2=\xi_E\cdot (\sigma_{|E})^*c_1(\ma N_{S/Z}^*)-(\sigma_{|E})^*c_2(\ma N_{S/Z}^*).$$
Therefore,
\begin{enumerate}[$i.$]\setlength\itemsep{4pt}
\item $(\sigma^*D)^3\cdot E=(\sigma_{|E})^*(D_{|S})^3=0$;
\item $(\sigma^*D)^2\cdot E^2=-\xi_E\cdot(\sigma_{|E})^*(D_{|S})^2=-(D_{|S})^2$;
\item $(\sigma^*D)\cdot E^3 =\xi_E^2\cdot(\sigma_{|E})^*(D_{|S})=D_{|S}\cdot c_1(\ma N_{S/Z}^*)$;
\item $E^4=-\xi_E^3=\dots=c_2(\ma N_{S/Z}^*)-c_1(\ma N_{S/Z}^*)^2$.
\end{enumerate}
\end{proof}

For the following lemma, we refer to the notation in \cite[$\S6.3$]{Debarre}.
\begin{lemma}\label{lem:cones}
Let $\ma C'\subseteq \ma C\subset \R^n$ be pointed convex cones of dimension $n$, with $\ma C'$ closed and polyhedral.
Assume that every facet of $\ma C'$ is contained in a facet of $\ma C$. Then, $\ma C'=\ma C$.
\end{lemma}
\begin{proof}
Let $F\subseteq \ma C'$ be a facet, so that $F=(\ph_F)^{\perp}\cap \ma C'$ for some $\ph_F\in (\ma C')^*$ vanishing along $F$. Then, $\ma C'$ is the set of vectors $v\in \R^n$ such that $\ph_F(v)\geq 0$ for all (finitely many) facets $F$ of $\ma C'$.
Since $\ph_F\in \ma C^*$ by hypothesis, we deduce that $\ma C\subseteq \ma C'$.
\end{proof}

\section{Setup and generality conditions}\label{ssetup}
\begin{setup}\label{setup}
Let $Z\subset \pr^5$ be a smooth cubic fourfold containing $m \leq 8$ planes $\Pi_1, \dots, \Pi_m$ pairwise meeting at distinct points $p_{ij}:=\Pi_i\cap \Pi_j$, and let $\sigma\colon Z_m\to Z$ be the blow-up along $\Pi_1\cup \dots\cup \Pi_m$.

By Corollary \ref{blowup-planes}, $Z_m$ is smooth and can be realised as the blow-up of $Z$ along $\Pi_1$, and subsequently along the strict transform of the remaining planes, one by one: the centre of each elementary blow-up is a smooth surface.
Since the order of the elementary blow-ups does not matter, $\sigma\colon Z_m\to Z$ admits $m!$ distinct factorizations
\begin{equation}\label{m!}
Z_m\to \dots \to Z_I \to \dots \to Z
\end{equation}
where $I\subseteq\{1,\dots,m\}$, $\sigma_I\colon Z_I\to Z$ is the blow-up along $\bigcup_{i\in I}\Pi_i$, and $\pi_I\colon Z_m \to Z_I$ is the blow-up along the strict transform of the remaining planes.
Set $Z_i:=Z_{\{1,\dots,i\}}$ for $1\leq i\leq m$.

For $1\leq i\leq m$, let $\mu_i\colon P_i\to \Pi_i$ be the blow-up along the points $\{p_{ij}\}_{j\neq i}$, i.e.\ the blow-up of $\pr^2$ at $m-1$ points.
The exceptional locus $\Exc(\sigma)$ is the union of $m$ exceptional divisors $E_i:=\sigma^{-1}(\Pi_i)$, each isomorphic to $\pr_{P_i}\bigl(\mu_i^*\ma N^*_{\Pi_i\slash Z}\bigr)$ and with tautological class $\ol_{E_i}(1)\simeq -{E_i}_{|E_i}$.
In fact, if we factor $\sigma\colon Z_m \to Z$ as
$$Z_m \xrightarrow{\sigma_m} Z_{m-1} \to \dots \to Z_1 \xrightarrow{\sigma_1} Z$$
with ${\sigma_i}$ the blow-up along the strict transform of $\Pi_i$, the strict transform in $Z_2$ of $\Exc({\sigma_1})$ is linearly equivalent to $\sigma_2^*\Exc({\sigma_1})$: this holds because $\Exc({\sigma_1})$ intersects the strict transform of $\Pi_2$, i.e.\ the blow-up of $\Pi_2$ at $p_{12}$, along its exceptional line which is also a fibre of $\Exc({\sigma_1})\to \Pi_1$, and similarly for all intermediate blow-ups and exceptional divisors.

It follows that
\begin{equation}\label{anticanonical}
-K_{Z_m} \simeq 3H- (E_1+\dots+ E_m)
\end{equation}
where $H:=\sigma^*\ol_Z(1)$.

Denote by $f_i$ the class of a fibre of $E_i \to P_i$. Then, $\sigma_{|E_i}\colon E_i \to \Pi_i$ is a fibration with general fibre $f_i\cong \pr^1$ and special fibres $E_i \cap E_j=\sigma^{-1}(p_{ij})\cong\pr^1\times \pr^1$ with rulings generated by $f_i$ and $f_j$. Thus,
\begin{equation}\label{intersection}
E_i\cdot f_i=-1, \quad E_i\cdot f_j=0
\end{equation}
for all $j\neq i$, and $E_i\cap E_j\cap E_k=\emptyset$ for all $1\leq i< j <k \leq m$.
\begin{figure}[h]\caption{\small Exceptional divisors of $\sigma\colon Z_m\to Z$.}\label{Ei-Ej}
\includegraphics[width=0.25\textwidth]{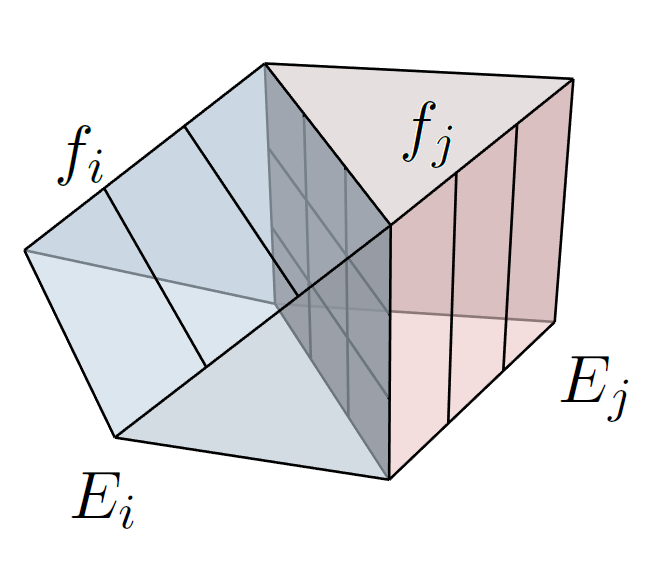}
\end{figure}

For each plane $\Pi_i$, the conormal bundle $\ma E:= \ma N^*_{\Pi_i\slash Z}$ sits in the exact sequence
\begin{equation}\label{M}
0 \to \ol_{\pr^2}(-3) \to \ol_{\pr^2}(-1)^{\oplus 3}\to \ma E \to 0
\end{equation}
given by the inclusion $\Pi_i\subset Z\subset\pr^5$, and we deduce from \eqref{M} that
\begin{equation}\label{chern_M}
\det (\ma E)\simeq \ol_{\pr^2}\quad \text{and} \quad c_2(\ma E)=3.
\end{equation}

Finally, we denote by $\w \Pi_{ijk}$ the plane generated by $\{p_{ij},\, p_{ik},\, p_{jk}\}$ for all $1\leq i <j<k \leq m$.

\medskip

By \textbf{abuse of notation}, $E_i$ is the exceptional divisor over the strict transform of $\Pi_i$ and $H\simeq\sigma_I^*\ol_Z(1)$ for all intermediate blow-ups $\sigma_I\colon Z_I \to Z$, $i\in I$.
\end{setup}

\medskip

We know that $Z_1$ and $Z_2$ are Fano by \cite[\S 11.7]{Casa25}. We describe their birational geometry in $\S$\ref{z1} and $\S$\ref{z2} as it will be essential in proving Proposition \ref{m=3} and Theorem \ref{thm:cremona}, the intermediate steps towards Theorems \ref{main:fano} and \ref{main:cone}. 
In $\S$\ref{gen_cond} we determine the necessary generality conditions $\boldsymbol{(\star)}$ so that $Z_m$ is Fano for $m\geq 3$.

\subsection{Birational geometry of $\boldsymbol{Z_1}$}\label{z1}
Let $\Pi_1\subset Z$ be as in Setup \ref{setup}.
The blow-up $\sigma\colon Z_1 \to Z$ along $\Pi_1$ is a smooth Fano fourfold with $\rho(Z_1)=2$, and it is well known that $Z_1$ has a quadric bundle onto $\pr^2$.
As this will be used later, we describe in detail its geometry in the lemma and remark that follow.

\begin{lemma}\label{quadric-bundle}
Let $\Pi_1\subset Z$  and $\sigma\colon Z_1 \to Z$ be as in Setup \ref{setup}. The divisor $H-E_1$ is globally generated and it induces a quadric bundle $\ph\colon Z_1\to\pr^2$.
\end{lemma}
\begin{proof}
Let $\tau\colon W_1\to\pr^5$ be the blow-up along $\Pi_1$. It is well known that $W_1$ has a $\pr^3$-bundle structure $$\phi\colon W_1\cong \pr_{\pr^2}\left(\ol(1)\oplus\ol^{\oplus 3}\right) \to \pr^2$$
satisfying $\tau^*\ol_{\pr^5}(1)-\Exc(\tau)\simeq\phi^*\ol_{\pr^2}(1)$; denote by $\ol_{W_1}(1)$ the tautological divisor class. Since $Z_1\simeq \phi^*\ol_{\pr^5}(3)-\Exc(\tau)\simeq -\frac{1}{2}K_{W_1}$, it follows that
$$Z_1\in |\ol_{W_1}(2)+\phi^*\ol_{\pr^2}(1)|$$
and so $\ph:=\phi_{|Z_1} \colon Z_1 \to \pr^2$ is a quadric bundle.
To conclude, note that $H-E_1$ is the restriction of $\tau^*\ol_{\pr^5}(1)-\Exc(\tau)$ to $Z_1$.
Thus, $H-E_1\simeq \ph^*\ol_{\pr^2}(1)$ and it is globally generated.
\end{proof}
\begin{remark}\label{quadric-bundle-geometry}
Setting as in Lemma \ref{quadric-bundle}: we describe in more details the geometry of the quadric bundle $\ph\colon Z_1\to\pr^2$ .
The fibres of $\phi\colon {W_1}\to \pr^2$ are the strict transform of the $3$-planes $\Lambda\subset\pr^5$ containing $\Pi_1$. Since $\Lambda$ is not tangent to $Z$ along $\Pi_1$ (otherwise $Z$ would be singular), the intersection $\Lambda \cap Z$ is a reducible cubic given by the union of $\Pi_1$ and a quadric surface $Q_{\Lambda}$ not containing $\Pi_1$. We deduce that $Z_1\subset{W_1}$ meets the strict transform of $\Lambda$ along the strict transform of $Q_{\Lambda}$, and these quadric surfaces are precisely the fibres of $\ph\colon Z_1\to\pr^2$: they are covered by lines $\ell\subset Z$ that intersect $\Pi_1$ at a point, i.e.\ such that $(H-E_1)\cdot \ell=0$.
\end{remark}

\subsection{Birational geometry of $\boldsymbol{Z_2}$}\label{z2}
Let $\Pi_1,\, \Pi_2\subset Z$ be as in Setup \ref{setup}. The blow-up $\sigma\colon Z_2 \to Z$ along $\Pi_1\cup \Pi_2$ is a smooth Fano fourfold with $\rho(Z_2)=3$, and it has exactly three elementary Mori contractions by \cite[\S 11.7]{Casa25}. Denote by $F_{12}$ the strict transform of the unique hyperplane section of $Z\subset\pr^5$ containing $\Pi_1\cup \Pi_2$. The elementary contractions of $Z_2$ are:
\begin{enumerate}[$i.$]\setlength\itemsep{4pt}
\item the blow-down ${\pi_{\{i\}}} \colon Z_2 \to Z_{\{i\}}$ of $E_j$ for $i\neq j\in\{1,2\}$, with supporting divisor $2H-E_i$;
\item the blow-down $\psi \colon Z_2 \to T$ of $F_{12}$, with supporting divisor $2H-E_1-E_2$.
\end{enumerate}
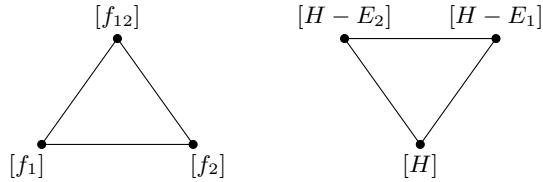
\begin{figure}[h]\caption{\small A section of $\NE(Z_2)$ on the left, and of $\Nef(Z_2)$ on the right.}\label{Cone-Z_2}
\vspace{5pt}
        \begin{tikzpicture}
        \draw (0,1.4) node[anchor=south]{\small$[f_{12}]$};
        \draw (1,0) node[anchor=125]{\small$[f_2]$};
        \draw (-1,0) node[anchor=55]{\small$[f_1]$};
        \draw[fill=black] (0,1.4) circle (1.5pt);
        \draw[fill=black] (1,0) circle (1.5pt);
        \draw[fill=black] (-1,0) circle (1.5pt);
        \draw (0,1.4) -- (1,0);
        \draw (0,1.4) -- (-1,0);
        \draw (1,0) -- (-1,0);
        \draw (4,0) node[anchor=north]{\small$[H]$};
        \draw (5,1.4) node[anchor=south]{\small$[H-E_1]$};
        \draw (3,1.4) node[anchor=south]{\small$[H-E_2]$};
        \draw[fill=black] (4,0) circle (1.5pt);
        \draw[fill=black] (5,1.4) circle (1.5pt);
        \draw[fill=black] (3,1.4) circle (1.5pt);
        \draw (4,0) -- (5,1.4);
        \draw (4,0) -- (3,1.4);
        \draw (5,1.4) -- (3,1.4); 
        \end{tikzpicture}
\end{figure}
\begin{remark}\label{H_12}
We now describe the geometry of $\psi \colon Z_2 \to T$, see \cite[\S 11.7]{Casa25}.
Let $\tau\colon W_2\to \pr^5$ be the blow-up along $\Pi_1$ and then along the strict transform of $\Pi_2$.
There exists a blow-up $\Psi\colon W_2 \to \pr_{\pr^2}(\ol\oplus\ol(1)^{\oplus 3})$ along a surface $V\cong\PP$, with exceptional divisor the transform of the unique hyperplane of $\pr^5$ containing $\Pi_1\cup\Pi_2$.
Then, $Z_2\subset W_2$ and the restriction
$$\psi:=\Psi_{|Z_2}\colon Z_2 \to T:=\Psi(Z_2)$$
is the blow-up along $V$ with $\Exc(\psi)=F_{12}=\Exc(\Psi)\cap Z_2$, where $T$ is a smooth Fano fourfold of index $2$ that is the double cover of $\pr^2\times\pr^2$ with branch divisor of bidegree $(2,2)$.
Such double cover $T$ of $\pr^2\times\pr^2$ is commonly known as a Verra fourfold \cite{CKKM}, and the rational map $\psi\circ\sigma^{-1}\colon Z \dashrightarrow T\subset \pr^9$ is induced by the linear system $|H^0(\ma I_{\Pi_1\cup\Pi_2/Z}\otimes\ol_Z(2))|$.
Finally, note that $F_{12}\simeq H-E_1-E_2$ and we denote by $f_{12}$ the class of a fibre of $\psi_{|F_{12}}\colon F_{12}\to V$. Figure \ref{Cone-Z_2} shows a section of the Mori and nef cones of $Z_2$.
\end{remark}

\subsection{Generality conditions}\label{gen_cond}
\begin{lemma}\label{condition:3bis}
If $Z_3$ is Fano, then $\Pi_1,\, \Pi_2,\, \Pi_3$ are not contained in a common hyperplane.
\end{lemma}
\begin{proof}
It follows from \eqref{m!} that $\sigma\colon Z_3 \to Z$ factors as
$$Z_3\xrightarrow{\pi_{\{1, 2\}}} Z_{2}  \xrightarrow{\sigma_{\{1, 2\}}} Z,$$
where $\sigma_{\{1,2\}} \colon Z_{\{1,2\}} \to Z$ is the blow-up along $\Pi_1\cup \Pi_2$, and $\pi_{\{1, 2\}}$ is the blow-up along the strict transform $P_3\cong\Bl_{\{2\, pts\}}\pr^2$ of $\Pi_3$.
Let $F_{12}$ be the exceptional divisor of the blow-up $\psi\colon Z_2 \to T$ from Remark \ref{H_12}.

If $\Pi_1,\, \Pi_2,\, \Pi_3$ are contained in a common hyperplane, then $P_3\subset F_{12}$.
We have that $P_3$ dominates the centre of $\psi$, otherwise $\psi(P_3)\cong\pr^1$ and $P_3$ would be Hirzebruch surface, a contradiction.
It follows that the strict transform $\bar\ell_3\subset P_3$ of the line $\ell_3:= \overline{p_{13}\,p_{23}}\subset \Pi_3$ is a $\psi$-fibre, and we deduce from
$$-K_{Z_3}\simeq \pi_{\{1, 2\}}^*(-K_{Z_2})-E_3$$
that $-K_{Z_3}\cdot\wi\ell_3=0$, where $\wi\ell_3$ is the strict transform of $\bar\ell_3\subset F_{12}\cap P_3$, and $Z_3$ is not Fano.
\end{proof}

\begin{lemma}\label{condition:m}
Assume that $Z_m$ is Fano. Then, $Z_{m-1}$ is Fano and $P_{i}$ is a smooth del Pezzo surface for all $1\leq i \leq m$.
\end{lemma}
\begin{proof}
It suffices to show the claim for $i=m$.
As above, we factor $\sigma\colon Z_m \to Z$ as
$$Z_m\xrightarrow{\pi_I} Z_{m-1}  \xrightarrow{\sigma_I} Z,$$
with $I=\{1, \dots, m-1\}$, so that $\pi_I$ is the blow-up along $P_m$.
By \eqref{chern_M} and adjunction we see that ${-K_{Z_m}}_{|E_m}$ is the tautological class of $\pr_{P_m}(-K_{P_m}\otimes \mu_m^*\ma E)$, which is ample since ${Z_m}$ is Fano. Thus, $\det\bigl(-K_{P_m}\otimes \mu_m^* \ma E\bigr)=-2K_{P_m}$ is ample and $P_m$ is a smooth del Pezzo surface.
Finally, we have ${-K_{Z_{m-1}}}_{|P_m}\simeq -K_{P_m}$ (again, by \eqref{chern_M} and adjunction) and it follows that $Z_{m-1}$ is Fano, see for instance \cite[Lemma 2.2]{HS25bis}.
\end{proof}

Recall from Setup \ref{setup} that $\w \Pi_{ijk}$ is the plane generated by $\{p_{ij},\, p_{ik},\, p_{jk}\}$, where $p_{ij}=\Pi_i\cap \Pi_j$.
\begin{lemma}\label{condition:3}
If $Z_3$ is Fano, then $\w\Pi_{123}\not\subset Z$.
\end{lemma}
\begin{proof}
Set $\w\Pi:=\w\Pi_{123}$ and let $Z$ be a general cubic fourfold in the linear system
$$|H^0(\ma I_{\Pi_1\cup \Pi_2 \cup \Pi_3 \cup \w \Pi/\pr^5}\otimes \ol_{\pr^5}(3))|.$$
Each $\Pi_i$ intersects $\w \Pi$ along the line $\overline{p_{ij}\,p_{ik}}$, so the restriction ${E_i}_{|\w P}$ along the strict transform $\w P$ of $\w \Pi$ is $\ol_{\pr^2}(1)$. It follows that $-K_{Z_3}\simeq3H-E_1-E_2-E_3$ is trivial along $\w P$, thus $Z_3$ is not Fano.
\end{proof}

\begin{generality}$\boldsymbol{(\star)}$
Let $\Pi_1, \dots, \Pi_m\subset Z$ be as in Setup \ref{setup}. We impose the following generality conditions for $3\leq m\leq8$:
\begin{enumerate}\setlength\itemsep{4pt}
\item[\textbf{G0.}] the planes $\{\Pi_i,\, \Pi_j,\, \Pi_k\}$ do not lie on a common hyperplane for all $1\leq i< j <k \leq m$;
\item[\textbf{G1.}] the blow-up $\mu_i\colon P_i\to \Pi_i$ at $\{p_{ij}\}_{j\neq i}$ is a smooth del Pezzo surface for all $1\leq i \leq m$, that is the points $\{p_{ij}\}_{j\neq i}\subset \Pi_i$ are in general position;
\item[\textbf{G2.}] the plane $\w \Pi_{ijk}$ is not contained in $Z$ for all $1\leq i< j <k \leq m$.
\end{enumerate}
\end{generality}

\noindent Note that the generality condition \textbf{G1} is automatically satisfied if $m=3$.

\subsection{Geometry of the configuration}\label{geom-conf}
Let $\Pi_1,\dots,\Pi_m\subset Z$ be as in Setup \ref{setup} with $m\geq3$, satisfying the generality conditions $\boldsymbol{(\star)}$. We set the following notation:
\begin{itemize}\setlength\itemsep{4pt}
\item $D_{ij}\subset \pr^5$ is the unique hyperplane $\Pi_i+\Pi_j$, it contains $\w\Pi_{ijk}$ for all $k\neq i,j$;
\item $\ell_i^{jk}$, or simply $\ell_i$ if $m=3$, is the line $\overline{p_{ij}\,p_{ik}}=D_{jk}\cap \Pi_i$;
\item ${\w\ell_{ijk}}^{ab}$ is the line $D_{ab}\cap\w\Pi_{ijk}$, spanned by the collinear points $\ell_i^{jk}\cap\ell_i^{ab}$, $\ell_j^{ik}\cap\ell_j^{ab}$ and $\ell_k^{ij}\cap\ell_k^{ab}$;
\item $Q_{\Lambda_i}$ is the quadric surface from the intersection $Z\cap\Lambda_i=\Pi_i\cup Q_{\Lambda_i}$, where $\Lambda_i$ is a $3$-plane containing $\Pi_i$ (see Remark \ref{quadric-bundle-geometry}).
\end{itemize}

\begin{figure}[h]\caption{\small Special lines in $\Pi_i,\,\Pi_j,\,\w\Pi_{ijk}$.}
\includegraphics[width=0.32\textwidth]{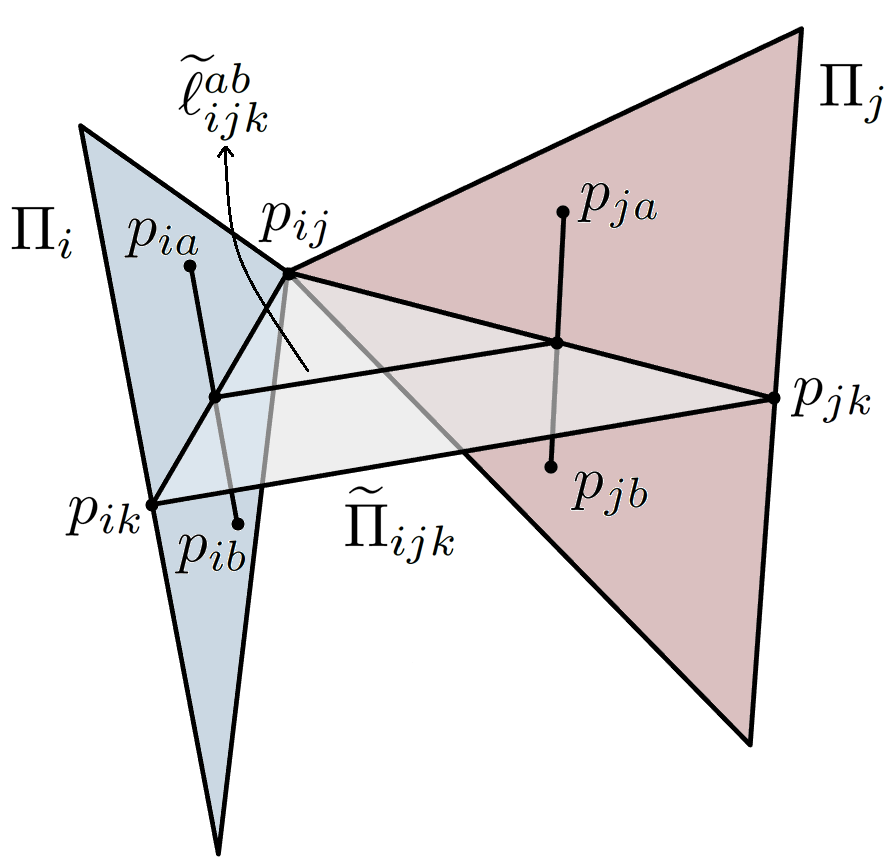}
\end{figure}

\begin{lemma}\label{lemma:geometry-3planes}
Let $\Pi_1,\, \Pi_2,\, \Pi_3\subset Z$ be as in Setup \ref{setup}, satisfying the generality conditions $\boldsymbol{(\star)}$.
Let $\ell\subset \pr^5$ be a line meeting $\Pi_i$ along a point for all $i\in\{1,2,3\}$, and set $\bar\Lambda_i:=D_{ij}\cap D_{ik}$.
Then:
\begin{enumerate}[i.]\setlength\itemsep{4pt}
\item $\ell\subset\w\Pi_{123}$, if the points $\ell\cap\Pi_i$ are distinct;
\item $\ell\not\subset Z$;
\item $Q_{\bar\Lambda_i}$ is a smooth quadric surface, and $Q_{\bar\Lambda_i}\cap \w\Pi_{123}=\ell_j\cup \ell_k$.
\end{enumerate}
\end{lemma}
\begin{proof}
If $\ell$ intersects $\Pi_1\cup\Pi_2\cup\Pi_3$ in three distinct points, then $\ell\subset \Pi_i+\Pi_j=D_{ij}$.  Thus, the generality condition \textbf{G0} implies that $\ell$ belongs to $D_{12}\cap D_{13}\cap D_{23}=\w\Pi_{123}$: this gives item \emph{i.}

We show \emph{ii.}\ and \emph{iii.}\ simultaneously.
First, note that $Z\cap \w \Pi_{123}=\ell_1\cup\ell_2\cup\ell_3$ by the generality condition \textbf{G2}, and that $\w\Pi_{123}\subset\bar\Lambda_i$. Thus, we deduce from
$$
Z\cap\w\Pi_{123}=(Z\cap\bar\Lambda_i) \cap \w\Pi_{123}=(\Pi_i\cap\w\Pi_{123})\cup(Q_{\bar\Lambda_i}\cap\w\Pi_{123})=\ell_i\cup(Q_{\bar\Lambda_i}\cap\w\Pi_{123})
$$
that $Q_{\bar\Lambda_i}\cap \w \Pi_{123}=\ell_j\cup\ell_k$.
To prove \emph{ii.}, it remains to show that $Z$ does not contain lines of type $\overline{p_{ij}\,q_k}$ for any $q_k\in \Pi_k\setminus \overline{p_{ik}\, p_{jk}}$: let us argue by contradiction and assume that a line $\ell:=\overline{p_{12}\,q_3}$ belongs to $Z$. It is clear from Remark \ref{quadric-bundle-geometry} that $\ell$ must be contained in a fibre of all three quadric bundles $\ph_i\colon Z_{\{i\}}\to \pr^2$, where $\sigma_{\{i\}} \colon Z_{\{i\}} \to Z$ is the blow-up along $\Pi_i$, that is $\ell\subset Q_{\Lambda_1}\cap Q_{\Lambda_2}\cap Q_{\Lambda_3}$ for some triple of $3$-planes $\Lambda_i$ containing $\Pi_i$.
Note that $p_{12}\in\Lambda_3$, so $\Lambda_3$ contains $\w\Pi_{123}$ and $\Lambda_3=\bar\Lambda_3$.
Moreover, $Q_{\bar\Lambda_3}$ contains the three lines $\ell,\, \ell_1,\,\ell_2$ passing through $p_{12}$, thus it is a singular quadric surface.
In particular, either $Q_{\bar\Lambda_3}$ is a quadric cone with vertex $p_{12}$ or it is the union of planes $\ell +\ell_1$ and $\ell +\ell_2$, and $Q_{\bar\Lambda_3}$ intersects $\Pi_3$ along a smooth conic through $\{p_{13},\,p_{23},\,q_3\}$ or along the lines $\overline{p_{13}\,q_3}\cup \overline{p_{23}\,q_3}$, respectively.
Each irreducible component of $Q_{\bar\Lambda_3}\setminus (\ell_1\cup\ell_2)$ is covered by lines through $p_{12}$ that intersect $\Pi_i$ at a point for all $i\in\{1,2,3\}$, hence $Q_{\bar\Lambda_3}$ is a fibre of $\ph_i\colon Z_{\{i\}}\to \pr^2$ for all $i\in\{1,2,3\}$.
We conclude that there exists a $3$-plane containing $\Pi_1\cup Q_{\bar\Lambda_3}$, thus containing $\Pi_1\cup \Pi_3$, which is impossible: this gives \emph{ii.} To conclude \emph{iii.}, note that if $Q_{\bar\Lambda_i}$ is singular, then it contains a line $\overline{p_{jk}\,q_i}$ as above, a contradiction.
\end{proof}

\begin{lemma}\label{4planes}
Let $\Pi_1,\dots, \Pi_4$ be as in Setup \ref{setup}, satisfying the generality condition \textbf{G0}. Then, $\Pi_i$ is disjoint from $\w\Pi_{abc}$ for $i\neq a,b,c$.
\end{lemma}
\begin{proof}
Suppose $\w\Pi_{123}\cap \Pi_4\neq\emptyset$. Since $\w\Pi_{123}$ is generated by $\{p_{12},\, p_{13},\, p_{23}\}$ and $\Pi_4$ is generated by $\{p_{14},\, p_{24},\, p_{34}\}$, we have that the six points $\{p_{ij}\}$ are linearly dependent, i.e.\ $\Pi_1,\dots, \Pi_4$ are contained in a common hyperplane, a contradiction.
\end{proof}

\section{The blow-up of the cubic fourfold along planes}\label{fano}
In this section we prove the main results of the paper, that is Theorems \ref{main:fano} and \ref{main:cone}.
This is done in three steps: in $\S$\ref{3} we show Theorem \ref{main:fano} for $m=3$, in $\S$\ref{sub:cremona} we prove the existence of the Cremona-like transformation from Theorem \ref{main:cremona}, and we treat the case $m \leq 8$ in $\S$\ref{8}.

In particular, $\S$\ref{8} is devoted to the proof of Theorem \ref{m=8}: we establish the correspondence between $\sigma \colon Z_m\to Z$ and the blow-up $g\colon S_m \to \pr^2$ along $m$ points in general position, and we show that $Z_m$ and $S_m$ share the same birational geometry.

Notation as in Section \ref{ssetup}.

\subsection{The blow-up along three planes}\label{3}
\begin{proposition}\label{m=3}
Let $\Pi_1,\, \Pi_2,\, \Pi_3\subset Z$ be as in Setup \ref{setup}, satisfying the generality conditions $\boldsymbol{(\star)}$. Then, the blow-up $\sigma \colon Z_3 \to Z$ along $\Pi_1\cup \Pi_2\cup \Pi_3$ is a smooth Fano fourfold.
\end{proposition}
\begin{proof}
Set $X:=Z_3$.
Recall from Setup \ref{setup} that $\sigma \colon X \to Z$ has three factorizations
$$X \xrightarrow{\pi_{\{i\}}} Z_{\{i\}} \xrightarrow{\sigma_{\{i\}}} Z$$
where $\sigma_{\{i\}} \colon Z_{\{i\}} \to Z$ is the blow-up along $\Pi_i$ and $\pi_{\{i\}} \colon X \to Z_{\{i\}}$ is the blow-up along the strict transform of the remaining planes.
By Lemma \ref{quadric-bundle}, $Z_{\{i\}}$ admits a quadric bundle $\ph_i\colon Z_{\{i\}}\to\pr^2$ induced by the divisor $\sigma_{\{i\}}^*\ol_Z(1)-E_i$, equivalent to the strict transform of a general hyperplane section of $Z\subset \pr^5$ containing $\Pi_i$.
It follows that
$$H_i:=\pi_{\{i\}}^*(\sigma_{\{i\}}^*\ol_Z(1)-E_i)\simeq H-E_i$$
is a globally generated divisor on $X$ supporting the contraction $\omega_i:=\ph_i\circ\pi_{\{i\}}\colon X \to \pr^2$, that is the blow-down of $\{E_j\}_{j\neq i}$ followed by $\ph_i$.
Then, \eqref{anticanonical} yields
\begin{equation}\label{eq_m=3}
-K_X \simeq H_1+H_2+H_3
\end{equation}
and we deduce that $-K_X$ is globally generated, thus semiample.
By Lemma \ref{ampleness}, it suffices to show that $-K_X$ is strictly nef.

Assume that there exists an irreducible curve $C\subset X$ such that $-K_X\cdot C=0$, so $H_i$ vanishes along $C$ by \eqref{eq_m=3} and $C$ is contracted by $\omega_i$ for all $i\in\{1,2,3\}$.
The curves contracted by $\pi_{\{i\}}$ have positive anticanonical degree, hence $\pi_{\{i\}}(C)$ is an irreducible curve contained in a $\ph_i$-fibre for all $i\in\{1,2,3\}$.
It follows from Remark \ref{quadric-bundle-geometry} that $\sigma(C)$ is an irreducible curve belonging to $Q_{\Lambda_1}\cap Q_{\Lambda_2}\cap Q_{\Lambda_3}$ for some triple of $3$-planes $\Lambda_i$ containing $\Pi_i$, and so $\Lambda_1\cap\Lambda_2\cap\Lambda_3$ is at least one-dimensional.
Note that $\dim(\Lambda_1\cap\Lambda_2\cap\Lambda_3)\leq2$ since the $\Lambda_i$'s cannot coincide, and recall that $Z\cap \w \Pi_{123}=\ell_1\cup\ell_2\cup\ell_3$ by the generality condition \textbf{G2}.

\emph{First case: $\Lambda_1\cap\Lambda_2\cap\Lambda_3$ is a line.}
If $\Lambda_1\cap\Lambda_2\cap\Lambda_3$ is a line $\ell$, then $\ell=\sigma(C)\subset Z$ intersects $\Pi_i$ for all $i\in\{1,2,3\}$, and Lemma \ref{lemma:geometry-3planes}\emph{ii.}\ implies that $\ell$ is one of the $\ell_i$'s, say $\ell_1$.
Note that $\Lambda_i=\Pi_i+\ell_1$ contains $\w\Pi_{123}$ for $i\in\{2,3\}$: we show that $\w\Pi_{123}\subset\Lambda_1$ as well, so that $\Lambda_1\cap\Lambda_2\cap\Lambda_3=\w\Pi_{123}$ gives a contradiction.
We have $\ell_1\subset \Pi_1\cap Q_{\Lambda_1}$, which implies that $\Lambda_1$ is tangent at $Z$ along $\ell_1$, and so $\Lambda_1$ is contained in $T_pZ$ for all $p\in \ell_1$.
Since we clearly have $T_{p_{ij}}Z=D_{ij}$ by construction, we deduce that $\Lambda_1=D_{12}\cap D_{13}\supset\w\Pi_{123}$.

\emph{Second case: $\Lambda_1\cap\Lambda_2\cap\Lambda_3$ is a plane.}
We have $\Lambda_1\cap\Lambda_2\cap\Lambda_3=\w \Pi_{123}$ by Lemma \ref{lemma:geometry-3planes}\emph{i.}, so $\Lambda_i=\bar\Lambda_i$ and $Q_{\bar\Lambda_i}\cap \w \Pi_{123}=\ell_j\cup\ell_k$ by Lemma \ref{lemma:geometry-3planes}\emph{iii.}
Thus,
$$\sigma(C)\subseteq Q_{\Lambda_1}\cap Q_{\Lambda_2}\cap Q_{\Lambda_3}=(Q_{\Lambda_1}\cap Q_{\Lambda_2}\cap Q_{\Lambda_3})\cap\w\Pi_{123} = \{p_{12},\, p_{13},\, p_{23}\},$$
a contradiction.
\end{proof}

\begin{remark}\label{re:line-types}
Let $\Pi_1,\, \Pi_2,\, \Pi_3\subset Z$ be as in Setup \ref{setup}. A line $\ell$ in $Z$ is said to be of \emph{first type} if there are no linear spaces $\Lambda\cong\pr^3$ that are tangent to $Z$ along $\ell$; otherwise $\ell$ is said to be of \emph{second type} (see \cite[Lemma 6.7]{CG72}).
If $\ell_i$ is of first type (resp.\ of second type), then $\ma N_{\ell_i/Z}^*\cong\ol_{\pr^1}^{\oplus 2}\oplus\ol_{\pr^1}(-1)$ (resp.\ $\ol_{\pr^1}(1)\oplus\ol_{\pr^1}(-1)^{\oplus 2}$) by \cite[Proposition 6.19]{CG72} and it follows that ${\ma E}_{|\ell_i}:=(\ma N^*_{\Pi_i\slash Z})_{|\ell_i}\cong\ol_{\pr^1}^{\oplus 2}$ (resp.\ $\ol_{\pr^1}(1)\oplus\ol_{\pr^1}(-1)$).
\end{remark}

Proposition \ref{m=3}, together with Lemmas \ref{condition:3bis} and \ref{condition:3}, shows that $Z_3$ is Fano if and only if the generality conditions \textbf{G0} and \textbf{G2} hold, and we deduce the following.
\begin{corollary}\label{co:line-types}
Let $\Pi_1,\, \Pi_2,\, \Pi_3\subset Z$ be as in Setup \ref{setup}. Then, the blow-up $\sigma \colon Z_3 \to Z$ along $\Pi_1\cup \Pi_2\cup \Pi_3$ is Fano if and only if the lines $\{\ell_1,\,\ell_2,\,\ell_3\}$ are of first type.
\end{corollary}
\begin{proof}
Set $X:=Z_3$.
It follows by adjunction that ${-K_X}_{|E_i}$ is the tautological divisor of $\pr_{P_i}(-K_{P_i}\otimes \mu_i^*\ma E)$.
If $\ell_i$ is of first type, the restriction of $-K_X$ to $\pi_{\{j, k\}}^{-1}(\overline \ell_i)$, where $\bar\ell_i$ is the strict transform of $\ell_i$ in $Z_{\{j, k\}}$, is the tautological divisor of $\pr_{\bar\ell_i}(\ol_{\pr^1}(1)^{\oplus 2})$, which is ample; otherwise, it is the tautological divisor of $\pr_{\bar\ell_i}(\ol_{\pr^1}(2)\oplus\ol_{\pr^1})$ and it is not ample.

If the generality condition \textbf{G0} does not hold, then we see from the proof of Lemma \ref{condition:3bis} that $-K_X\cdot\wi\ell_i=0$, a contradiction to $\ell_i$ being of first type since $\wi\ell_i\subset \pi_{\{j, k\}}^{-1}(\overline \ell_i)$.

If the generality condition \textbf{G2} does not hold, then the intersection of $\bar\Lambda_i:=D_{ij}\cap D_{ik}$ and $Z$ contains $\Pi_i\cup\w\Pi_{123}$, thus $\bar \Lambda_i$ is tangent to $Z$ along $\ell_i=\Pi_i\cap\w\Pi_{123}$, i.e.\ $\ell_i$ is of second type.
\end{proof}

\subsection{A Cremona-like transformation.}\label{sub:cremona}
Let $\Pi_1,\, \Pi_2,\, \Pi_3\subset Z$ be as in Setup \ref{setup}, satisfying the generality conditions $\boldsymbol{(\star)}$, and let $\sigma \colon Z_3 \to Z$ be the blow-up along $\Pi_1\cup \Pi_2\cup \Pi_3$.
Then, $Z_3$ is Fano by Proposition \ref{m=3}. 

We factor $\sigma \colon Z_3 \to Z$ as
$$Z_3 \xrightarrow{\pi_{\{i,j\}}} Z_{\{i,j\}} \xrightarrow{\sigma_{\{i,j\}}} Z$$
where $\sigma_{\{i,j\}} \colon Z_{\{i,j\}} \to Z$ is the blow-up along $\Pi_i\cup \Pi_j$ and $\pi_{\{i,j\}} \colon Z_3 \to Z_{\{i,j\}}$ is the blow-up along the strict transform $P_k$ of the remaining plane $\Pi_k$, $i< j <k \in\{1,2,3\}$.

Let $F_{ij}\subset Z_{\{i,j\}}$ be the exceptional divisor of the blow-up $\psi_{ij}\colon Z_{\{i,j\}} \to T_{ij}$ with centre $V_{ij}\cong\PP$ from Remark \ref{H_12}, that is the strict transform of the unique hyperplane section of $Z\subset\pr^5$ containing $\Pi_i\cup\Pi_j$.
Then, $2H-E_i-E_j$ (as a divisor in $Z_{\{i,j\}}$) is a supporting divisor of $\psi_{ij}$, $F_{ij}\cong \pr_{V_{ij}}(\ma N^*_{V_{ij}/T_{ij}})$ and ${\psi_{ij}}_{|F_{ij}}\colon F_{ij} \to V_{ij}$ is the natural projection.

The intersection of $F_{ij}$ and $P_k$ is the strict transform $\bar\ell_k\subset P_k$ of the line $\ell_k\subset \Pi_k$.
Thus, $(2H-E_i-E_j)\cdot \bar\ell_k=0$ and $\bar\ell_k\subset F_{ij}$ is a non-trivial fibre of $\psi_{ij}$; see Figure \ref{H_ij-E_k}.

Let now $H_{ij}\subset Z_3$ be the strict transform of $F_{ij}$. Since $F_{ij}$ intersects $P_k$ along a curve that is a fibre of its ruling over $V_{ij}$, we have
\begin{equation}\label{Hij}
H_{ij}\simeq \pi_{\{i,j\}}^*F_{ij}\simeq H-E_i-E_j\quad \text{and} \quad H_{ij}\cong\pr_{P_{ij}}(\eta_{ij}^*\ma N^*_{V_{ij}/T_{ij}}),
\end{equation}
where $\eta_{ij}\colon {P_{ij}}\to {V_{ij}}$ is the blow-up at a point, i.e.\ $P_{ij}\cong\Bl_{\{2\, pts\}}\pr^2$.
We run the two-rays-game for $\kappa_{ij}:=\psi_{ij}\circ{\pi_{\{i,j\}}}\colon Z_3 \to T_{ij}$, and we see that the contraction of $H_{ij}$ exists by \cite[Theorem 1.2]{Wisn91I}: it is a smooth blow-up $\varepsilon_{ij}\colon Z_3 \to M_{ij}$ with centre $P_{ij}\subset M_{ij}$.
We denote by $f_{ij}$ the general fibre of the projection ${\varepsilon_{ij}}_{|H_{ij}}\colon H_{ij}\to P_{ij}$.
$$\xymatrix{
& Z_3 \ar[ld]_{\pi_{\{i,j\}}} \ar[rd]^{\varepsilon_{ij}} \ar[dd]^{\kappa_{ij}} & \\
Z_{\{i,j\}}\ar[rd]_{\psi_{ij}} & & M_{ij} \ar[ld] \\
& T_{ij} & 
}$$

\begin{lemma}\label{H_ij:intersection}
Let $H_{ij}\subset Z_3$ as above. Then, $H_{ij}\cap H_{ik}\cong\PP$ with rulings generated by $f_{ij}$ and $f_{ik}$, and $H_{ij}\cap E_k\cong\PP$ with rulings generated by $f_{ij}$ and $f_{k}$, for all $i< j <k \in\{1,2,3\}$.
In particular,
\begin{gather*}
H_{ij}\cdot f_{ij}=-1, \quad H_{ij}\cdot f_{ik}=H_{ij}\cdot f_{jk}=0, \\
H_{ij}\cdot f_k=E_k\cdot f_{ij}=0, \quad H_{ij}\cdot f_i=E_i\cdot f_{ij}=1
\end{gather*}
and $H_{12}\cap H_{13}\cap H_{23}=\emptyset$.
\end{lemma}
\begin{proof}
Let $\bar\Lambda_i:=D_{ij}\cap D_{ik}$.
We have $\bar\Lambda_i\cap Z=\Pi_i\cup  Q_{\bar\Lambda_i}$, and $Q_{\bar\Lambda_i}$ is a smooth quadric surface with rulings generated by $\ell_j$ and $\ell_k$ by Lemma \ref{lemma:geometry-3planes}\emph{iii.}
Since $H_{ij}$ is the strict transform of $D_{ij}\cap Z$, we deduce that $H_{ij}$ intersects $H_{ik}$ along the strict transform $Q_{i}$ of $Q_{\bar\Lambda_i}$.
It follows that the rulings of $Q_{i}$ are generated by $f_{ij}$ and $f_{ik}$.

The statement about $H_{ij}$ and $E_k$ follows from $F_{ij}\cap P_k=\bar\ell_k$, and from the isomorphism ${\ma E}_{|\ell_i}\cong\ol_{\pr^1}^{\oplus 2}$ (see Remark \ref{re:line-types} and Corollary \ref{co:line-types}), so that $\pi_{\{i,j\}}^{-1}(\bar\ell_k)\cong\PP$.

Lastly, the statement about $H_{ij}$ and $E_i$ follows easily from $\Pi_i\subset D_{ij}\cap Z$.
\end{proof}
\begin{figure}[h]\caption{\small $F_{ij},P_k\subset Z_{\{i,j\}}$ and $H_{ij},E_k\subset Z_3$.}\label{H_ij-E_k}
\includegraphics[width=0.64\textwidth]{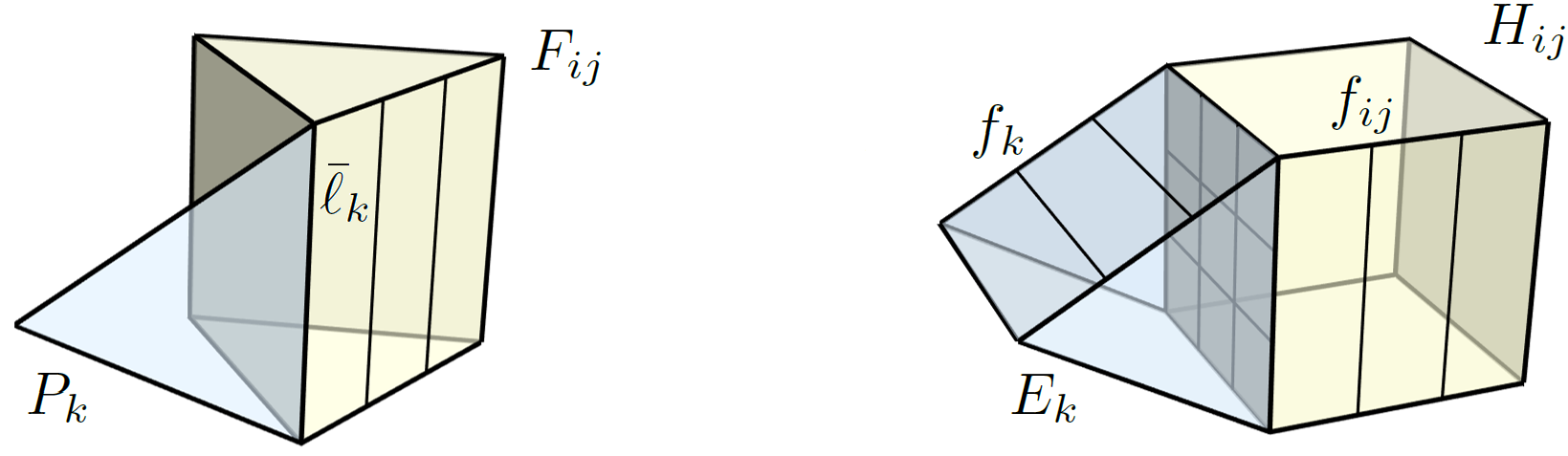}
\end{figure}

\begin{thm}\label{thm:cremona}
Let $\Pi_1,\, \Pi_2,\, \Pi_3\subset Z$ be as in Setup \ref{setup}, satisfying the generality conditions $\boldsymbol{(\star)}$, and let $\sigma \colon Z_3 \to Z$ be the blow-up along $\Pi_1\cup \Pi_2\cup \Pi_3$.
There exist a smooth cubic fourfold $Z'$ and a rational map $\zeta\colon Z\dashrightarrow  Z'$ such that $\sigma':=\zeta\circ\sigma\colon Z_3 \to Z'$ is the blow-up along planes $\Pi_1'\cup\Pi_2'\cup\Pi_3'\subset Z'$ as in Setup \ref{setup}, satisfying the generality conditions $\boldsymbol{(\star)}$.

Moreover, $\Exc(\sigma')=H_{12}\cup H_{13}\cup H_{23}$ and $(\sigma')^*\ol_{Z'}(1)\simeq2H-E_1-E_2-E_3$.
\end{thm}
\begin{proof}
Set $X:=Z_3$.
Recall that $X$ is Fano by Proposition \ref{m=3}, and set $H':=2H-E_1-E_2-E_3$.
We have that $H'$ is nef and big with $(H')^4=3$ by Step \ref{H'}, hence $|nH'|$ induces a birational contraction $\sigma' \colon X \to Z'$ for $n\gg0$.
We deduce from Lemma \ref{H_ij:intersection} that $H'\cdot f_{ij}=0$, so $H_{ij}$ is a component of $\Exc(\sigma')$ and $\sigma'$ factors through the contraction of $H_{ij}$ for all $i< j \in\{1,2,3\}$.
It follows from $\rho(X)=4$ that $\rho(Z')=1$ and that $\sigma' \colon X \to Z'$ is the contraction of $H_{12}\cup H_{13}\cup H_{23}$.
The $P_{ij}$'s are isomorphic to the blow-up of $\pr^2$ at two points, so $\Pi_k':=\sigma' (H_{ij})\cong\pr^2$, $\Pi_i'$ intersects $\Pi_j'$ along the point $\sigma'(H_{jk}\cap H_{ik})$, and $\sigma'$ is the blow-up along  $\Pi'_1\cup \Pi'_2\cup \Pi'_3$.
Moreover, each elementary blow-up in a factorization of $\sigma'$ is a smooth blow-up, thus $Z'$ is smooth.

Finally, we show in Step \ref{Z'} that $Z'$ is a cubic fourfold.
It follows that $\Pi_1',\,\Pi_2',\,\Pi_3'\subset Z'$ satisfy the generality conditions $\boldsymbol{(\star)}$ by Lemmas \ref{condition:3bis} and \ref{condition:3}, and the rational map $\zeta\colon Z\dashrightarrow Z'$ is realised as $\sigma'\circ \sigma^{-1}$.

\begin{step}\label{H'}
$H'$ is nef and big, and $(H')^4=3$.
\end{step}
\begin{proof}
We have $H'\simeq H_{i}+H_{jk}$, where $H_{jk}\simeq H-E_j-E_k$ from \eqref{Hij} and $H_{i}\simeq H-E_i$ is nef (see proof of Proposition \ref{m=3}). If there exists a curve $C$ such that $H'\cdot C<0$, then $H_{ij}\cdot C<0$ for all $i< j \in\{1,2,3\}$. Therefore, $C$ is contained in $H_{12}\cap H_{13}\cap H_{23}$, a contradiction by Lemma \ref{H_ij:intersection}.
Finally, $(H')^4=3$ by \eqref{chern_M} and Lemma \ref{conti} applied to the blow-up along $\Pi_1$, and subsequently along the strict transform of $\Pi_2$ and $\Pi_3$, with $D=2H$, $2H-E_1$, and $2H-E_1 - E_2$ respectively.
\end{proof}

\begin{step}\label{Z'}
$Z'$ is a cubic fourfold.
\end{step}
\begin{proof}
By the classification \cite[Theorem 3.3.1]{FanoEMS}, it suffices to show that $Z'$ is a smooth del Pezzo fourfold of degree 3. Since $X$ is Fano, we have that $Z'$ is rationally connected, so $\rho(Z')=1$ implies that $Z'$ is Fano as well.
We have
$$-K_X\simeq (\sigma')^*(-K_{Z'})-H_{12}-H_{13}-H_{23},$$
so \eqref{anticanonical} gives $(\sigma')^*(-K_{Z'})\simeq 3H'$. We deduce that $-K_{Z'}\simeq \ol_{Z'}(3)$, where $\ol_{Z'}(1)$ is the ample generator of $\Pic(Z')$, and so $Z'$ is a del Pezzo fourfold. We conclude from $(\sigma')^*\ol_{Z'}(1)\simeq H'$ and $(H')^4=3$.
\end{proof}
\end{proof}

\begin{remark}\label{re:cremona}
Let $m\geq3$ and $\Pi_1,\dots,\Pi_m\subset Z$ be as in Setup \ref{setup}, satisfying the generality conditions $\boldsymbol{(\star)}$, and let $\zeta_{ijk}\colon Z \dashrightarrow Z'$ be the rational map from Theorem \ref{thm:cremona} with respect to the planes $\{\Pi_i,\,\Pi_j,\,\Pi_k\}$, for $1\leq i<j<k\leq m$.
Then, $\zeta_{ijk}$ corresponds to the rational map induced by $|H^0(\ma I_{\Pi_i\cup\Pi_j\cup\Pi_k/Z}\otimes\ol_Z(2))|$, and it is the restriction to $Z$ of the Cremona transformation of $\pr^5$ defined by $|H^0(\ma I_{\Pi_i\cup\Pi_j\cup\Pi_k/\pr^5}\otimes\ol_{\pr^5}(2))|$.
In fact, \eqref{horace} applied to $W=\pr^5$, $X=\Pi_i\cup\Pi_j\cup\Pi_k$ and $D=Z$ yields
\begin{equation*}
0 \to \ol_{\pr^5}(-1)\to \ma I_{\Pi_i\cup\Pi_j\cup\Pi_k/\pr^5}\otimes\ol_{\pr^5}(2) \to\ma I_{\Pi_i\cup\Pi_j\cup\Pi_k/Z}\otimes\ol_Z(2) \to 0,
\end{equation*}
so that the restriction
$
H^0(\ma I_{\Pi_i\cup\Pi_j\cup\Pi_k/\pr^5}\otimes\ol_{\pr^5}(2)) \to H^0(\ma I_{\Pi_i\cup\Pi_j\cup\Pi_k/Z}\otimes\ol_Z(2))
$
is an isomorphism.
We call $\zeta_{ijk}$ a \emph{Cremona-like transformation} of $Z$ centred at $\Pi_i\cup\Pi_j\cup\Pi_k$.
$$\xymatrix{
& Z_m \ar[ld]_{\sigma} \ar[rd]^{\sigma'} & \\
Z\ar@{-->}[rr]_{\zeta_{ijk}} & & Z'
}$$

We now describe the geometry of $\zeta_{ijk}$: for simplicity, let us fix indices and consider $\zeta_{123}$. Then, $\sigma'\colon Z_m \to Z'$ is the blow-up along planes $\Pi_1'\cup\dots\cup\Pi_m'\subset Z$ as in Setup \ref{setup}, satisfying the generality conditions $\boldsymbol{(\star)}$, where $\Pi_i'=\sigma'(H_{jk})$ for $1\leq i<j<k\leq 3$.
It is not difficult to see that the restriction of $\zeta_{123}$ to $\Pi_s$ for $4\leq s \leq m$
$$
{\zeta_{123}}_{|\Pi_s} \colon \Pi_s \dashrightarrow \Pi_s'
$$
is the standard Cremona transformation of $\Pi_s \cong\pr^2$ centred at $\{p_{1s},\,p_{2s},\,p_{3s}\}$, that is the rational map induced by $|H^0(\ma I_{\{p_{1s},\,p_{2s},\,p_{3s}\}/\Pi_s}\otimes\ol_{\Pi_s}(2))|$.
\end{remark}

\subsection{The blow-up along $\boldsymbol{m\leq8}$ planes}\label{8}
Let $S_m$ be a del Pezzo surface of degree $9-m$ obtained as the blow-up $g\colon S_m\to \pr^2$ along $m\leq 8$ points $\{q_1, \dots, q_m\}$ in general position; we denote by $e_i$ the $g$-exceptional line over $q_i$, by $l\simeq g^*\ol_{\pr^2}(1)$ the transform of a general line, and by $l_{ij}$ the transform of the line $\overline{q_i\,q_j}$.
We refer to \cite{Dolg83, Dolg12} for an account of del Pezzo surfaces and Cremona isometries; here we collect the notions relevant to the current section.
\begin{definition}
A Cremona isometry on $S_m$ is an automorphism of $\Pic(S_m)$ that preserves the intersection form on $\Pic(S_m)$ and that fixes the anticanonical divisor $-K_{S_m}$.
\end{definition}
The group of Cremona isometries ${\rm Cris}({S_m})$ is generated by the automorphisms of $\Pic(S_m)$ induced by permutations of the blown-up points, together with the automorphisms \eqref{cremona:isometry} induced by the standard Cremona transformations of $\pr^2$ centred at $\{q_i, q_j, q_k\}$ when $m\geq3$.
\begin{equation}\label{cremona:isometry}
\begin{gathered}
\Pic({S_m}) \overset{\cong}{\longrightarrow} \Pic({S_m})\\
l \mapsto l'\simeq 2l-e_i-e_j-e_k\\
e_i \mapsto e_i':=l_{jk}\simeq l-e_j-e_k \\
e_t \mapsto e_t':=e_t \quad \forall t \neq i,j,k
\end{gathered}
\end{equation}
Once a presentation $g\colon S_m\to \pr^2$ is fixed, ${\rm Cris}({S_m})$ can be identified with a subgroup of the Cremona group ${\rm Bir}(\pr^2)$, see \cite[$\S8.2.8$]{Dolg12}.
Moreover, ${\rm Cris}({S_m})$ is naturally isomorphic to the Weyl group associated with the root system of $S_m$ \cite[Theorem 1]{Dolg83} and it acts transitively on the geometric markings of ${S_m}$ \cite[Corollary 1]{Dolg83}, namely on the sets of $m$ pair-wise disjoint $(-1)$-curves of ${S_m}$. In other words, every Cremona isometry sends $\{e_1,\dots, e_m\}$ to the exceptional curves $\{e_1',\dots, e_m'\}$ of some blow-up $g'\colon {S_m}\to \pr^2$ at $m$ points in general position, and $l$ to $l'\simeq (g')^*\ol_{\pr^2}(1)$.

Cremona isometries extend naturally to automorphisms of $\ma N^1({S_m})$: it follows that the action of ${\rm Cris}({S_m})$ on $\ma N^1({S_m})$ is invariant on the effective cone $\Eff({S_m})=\NE({S_m})$, and it is transitive on the set of its extremal rays when $m\geq 3$.

Let now $\Pi_1,\dots,\Pi_m\subset Z$ be as in Setup \ref{setup}, and let $\sigma \colon Z_m \to Z$ be the blow-up of $Z$ along $\Pi_1\cup\dots\cup\Pi_m$.
\begin{lemma}\label{lemma:h}
A general line in $Z$ is disjoint from $\Pi_1\cup\dots\cup\Pi_m$.
\end{lemma}
\begin{proof}
The variety of lines contained in a smooth cubic fourfold has dimension $4$ by \cite[Lemma 7.7]{CG72}, so it suffices to show that the family of lines in $Z$ intersecting a given plane $\Pi\subset Z$ has dimension $\leq 3$.

Recall from Lemma \ref{quadric-bundle} and Remark \ref{quadric-bundle-geometry} that the blow-up $\sigma \colon Z_1\to Z$ along $\Pi$ has a quadric bundle $\ph\colon Z_1\to \pr^2$ such that a line in $Z$ is contracted by $\ph$ if and only if it intersects $\Pi$ along a point.
Now, since the family of lines contained in a general $\ph$-fibre is one-dimensional (i.e.\ a $\pr^1\sqcup\pr^1$ of lines in $\PP$), the family of lines contracted by $\ph$ is three-dimensional. This concludes the proof.
\end{proof}
Lemma \ref{lemma:h} implies that the strict transform $h$ of a general line in $Z$ is disjoint from the exceptional locus $\Exc(\sigma)=E_1\cup\dots\cup E_m$, thus $E_i\cdot h=0$ for all $i\in\{1,\dots,m\}$.
Note that $\ma N^1(Z_m)$ and $\ma N_1(Z_m)$ are generated respectively by the classes of $\{H, E_1, \dots, E_m\}$ and $\{h, f_1, \dots, f_m\}$; we collect in Table \ref{intersections} their intersection numbers.
\begin{table}[h] \caption{\small Intersection numbers on $Z_m$.}\label{intersections}
$\begin{array}{||c|c|c|c||}
\hline\hline
\rule{0pt}{2.5ex} \cdot & h & f_i & f_j \\
\hline\hline
\rule{0pt}{2.5ex} H & 1 & 0 & 0 \\
\hline
\rule{0pt}{2.5ex} E_i & 0 & -1 & 0 \\
\hline
\rule{0pt}{2.5ex} E_j & 0 & 0 & -1 \\
\hline\hline
\end{array}$
\end{table}
\begin{remark}\label{alpha}
Let $g\colon S_m\to \pr^2$ as above, and let $\Pic(S_m)\xrightarrow{\cong}\Pic(Z_m)$ be the isomorphism that sends $l$ to $H$ and $e_i$ to $E_i$ for all $i\in\{1,\dots,m\}$. Tensoring it with $\R$ over $\Z$, we get an isomorphism
$$\alpha\colon \ma N^1(S_m)\xrightarrow{\cong}\ma N^1(Z_m),$$
and we deduce from Table \ref{intersections} that the inverse of its dual
$$\beta:=(\alpha^*)^{-1}\colon \ma N_1(S_m)\xrightarrow{\simeq}\ma N_1(Z_m)$$
sends $[l]$ to $[h]$ and $[e_i]$ to $[f_i]$ for all $i\in\{1,\dots,m\}$.
Note that $\alpha$ and $\beta$ preserve the intersection forms, that is $\alpha(M)\cdot \beta(\gamma)=M\cdot \gamma$ for all $M\in \ma N_1({S_m})$ and $[\gamma]\in \ma N^1({S_m})$.
Therefore, we have $\alpha(-K_{S_m})\simeq -K_{Z_m}$.
\end{remark}

\begin{remdef}\label{cris_X}
Similarly to ${\rm Cris}(S_m)$, we denote by ${\rm Cris}(Z_m)$ the group generated by the automorphisms of $\Pic(Z_m)$ induced by permutations of the blown-up planes, together with the automorphisms \eqref{Cremona:Z} induced by the Cremona-like transformations $\{\zeta_{ijk}\colon Z \dashrightarrow Z'\}$ from Remark \ref{re:cremona} when $m\geq3$.
In fact, each $\zeta_{ijk}$ extends to an automorphism of $\Pic(Z_m)$
\begin{equation}\label{Cremona:Z}
\begin{gathered}
\zeta_{ijk} \colon \Pic(Z_m) \overset{\cong}{\longrightarrow} \Pic(Z_m)\\
H \mapsto H'\simeq 2H-E_i-E_j-E_k\\
E_i \mapsto E_i':=H_{jk,X}\simeq H-E_j-E_k \\
E_t \mapsto E_t':=E_t \quad \forall t \neq i,j,k
\end{gathered}
\end{equation}
that sends the exceptional divisors $\{E_1, \dots ,E_m\}$ of $\sigma \colon Z_m \to Z$ to the exceptional divisors $\{E_1', \dots ,E_m'\}$ of $\sigma'\colon Z_m \to Z'$ and $H\simeq\sigma^*\ol_Z(1)$ to $H'\simeq(\sigma')^*\ol_{Z'}(1)$.

It follows that for all $\zeta \in {\rm Cris}(Z_m)$ there exists a blow-up $\sigma_\zeta \colon Z_m \to Z_\zeta$ of a smooth cubic fourfold $Z_\zeta$ along $m$ planes as in Setup \ref{setup}, satisfying the generality conditions $\boldsymbol{(\star)}$.
Therefore, $\alpha$ induces an isomorphism between ${\rm Cris}(Z_m)$ and ${\rm Cris}({S_m})$, and the action of ${\rm Cris}(Z_m)$ on $\Pic(Z_m)$ corresponds to the action of ${\rm Cris}({S_m})$ on $\Pic({S_m})$.

Lastly, we denote by $\Exc(Z_m)$ the image under $\alpha$ of the set of $(-1)$-curves of ${S_m}$: $\Exc(Z_m)$ consists of the exceptional divisors of $\sigma_\zeta \colon Z_m \to Z_\zeta$ for all $\zeta \in {\rm Cris}(Z_m)$ when $m\geq3$, whereas $\Exc(Z_1)=\{E_1\}$ (see $\S$\ref{z1}) and $\Exc(Z_2)=\{E_1,\, E_2,\,F_{12}\}$ (see $\S$\ref{z2}).
\end{remdef}

\begin{thm}\label{m=8}
Let $\Pi_1,\dots,\Pi_m\subset Z$ be as in Setup \ref{setup}, satisfying the generality conditions $\boldsymbol{(\star)}$, and let ${S_m}$, $\alpha$ and $\beta$ be as in Remark \ref {alpha}.
Then, the blow-up $\sigma \colon Z_m \to Z$ along $\Pi_1\cup\dots\cup\Pi_m$ is a smooth Fano fourfold. Moreover:
\begin{enumerate}[i.]\setlength\itemsep{4pt}
\item the restriction of $\alpha$ to $\NE({S_m})$ yields an isomorphism
$$\NE({S_m})\xrightarrow{\simeq}\Eff(Z_m)$$
that sends each $(-1)$-curve $d$ of ${S_m}$ to an exceptional divisor $D\in \Exc(Z_m)$.
\item the restriction of $\beta$ to $\NE({S_m})$ yields an isomorphism
$$\NE({S_m})\xrightarrow{\simeq}\NE(Z_m)$$
that sends each $(-1)$-curve $d$ of ${S_m}$ to the class of a non-trivial fibre of the contraction of $D:=\alpha(d)$.
\end{enumerate}
When $m\geq2$, the extremal rays of $\Eff(Z_m)$ are generated by the divisors in $\Exc(Z_m)$, and the non-trivial fibres of the corresponding contractions generate the extremal rays of $\NE(Z_m)$.
\end{thm}
\begin{proof}
Set $X=Z_m$.
The statements easily follow from $\S$\ref{z1} and $\S$\ref{z2} for $m\in\{1,2\}$, so we may assume $m\geq 3$. We already know that $X$ is Fano for $m=3$ by Proposition \ref{m=3}.

We deduce from Remark \ref{alpha} and Remark - Definition \ref{cris_X} that $\Exc(X)$ consists of the exceptional divisors of the blow-ups $\sigma_\zeta \colon X \to Z_\zeta$ for all $\zeta \in {\rm Cris}(X)$, and that ${\rm Cris}(X)$ acts transitively on it.
Note that $-K_X$ is fixed by the action of ${\rm Cris}(X)$.

We start by proving item \emph{ii.}
Let $D:=\alpha(d)\in \Exc(X)$ be the exceptional divisor corresponding to a $(-1)$-curve $d$ of $S_m$, and let $g_D\colon X \to X_D$ be its contraction.
Then, $\beta(d)$ is the class of the non-trivial fibres of $g_D$ by Step \ref{extremal ray}, so it generates an extremal ray of $\NE(X)$.
Therefore, every extremal ray of $\beta(\NE({S_m}))\subseteq\NE(X)$ is an extremal ray of $\NE(X)$.
To see that equality holds, note that every facet of $\beta(\NE({S_m}))$ is contained in facets of $\NE(X)$ by Step \ref{facet}, thus
$\beta(\NE({S_m}))= \NE(X)$
by Lemma \ref{lem:cones}.

To show that $X$ is Fano, note that \emph{ii.}\ and Remark \ref{alpha} imply that $-K_X$ is positive on the extremal rays of $\NE(X)$, hence $-K_X$ is ample by Kleiman's criterion \cite[Theorem 1.27]{Debarre}.

Finally, \emph{ii.}\ implies that all exceptional divisors of $X$ appear in $\Exc(X)$. The cone generated by $\Exc(X)$ in $\ma N^1(X)$ coincides with $\alpha(\NE({S_m}))\subseteq\Eff(X)$, and equality holds by \cite[Proposition 1.11(2)]{HK00} and \cite[Corollary 1.3.2]{BCHM}).
This concludes the proof of item \emph{i.}

\begin{step}\label{extremal ray}
Let $d$ be a $(-1)$-curve of $S_m$.
Then, $\beta(d)$ is the class of the non-trivial fibres of the contraction $g_D\colon X \to X_D$ of $D:=\alpha(d)$.
\end{step}
\begin{proof}
Let $f_D$ be a non-trivial fibre of $g_D$: we prove that $f_D\equiv \beta(d)$. It follows from above that there exist $\zeta \in {\rm Cris}(X)$ and exceptional divisors $D_1, \dots, D_{m-1}\in \Exc(X)$ such that the blow-up $\sigma_\zeta \colon X \to Z_\zeta$ is the contraction of $\{D,D_1, \dots, D_{m-1}\}$, and so $D_i\cdot f_D=0$ for all $i\in\{1,\dots,m-1\}$ by \eqref{intersection} (cf.\ Table \ref{intersections}).
Therefore, $\sigma_\zeta$ corresponds via $\alpha$ to the contraction $g\colon {S_m} \to \pr^2$ of the $(-1)$-curves $\{d, d_1, \dots, d_{m-1}\}$ where $d_i:=\alpha^{-1}(D_i)$, and
\begin{gather*}
D\cdot f_D=-1, \quad D\cdot \beta(d)=d^2=-1, \\
K_X\cdot f_D=-1, \quad K_X\cdot \beta(d)=K_{S_m}\cdot d=-1, \\
D_i\cdot f_D=0 \ \forall i, \quad D_i\cdot \beta(d)=d_i\cdot d=0 \ \forall i.
\end{gather*}
To conclude, note that $\{K_X,D, D_1, \dots, D_{m-1}\}$ forms a basis of $\ma N^1 (X)$.
\end{proof}

\begin{step}\label{facet}
Every facet of $\beta(\NE({S_m}))$ is contained in facets of $\NE(X)$.
\end{step}
\begin{proof}
Each facet $F$ of $\NE(S_m)$ corresponds either to a blow-up ${S_m} \to \pr^2$, and $\beta(F)=\NE(\sigma_\zeta)$ for some $\sigma_\zeta \colon X \to Z_\zeta$, or to a conic bundle ${S_m} \to \pr^1$ that factors through a $\pr^1$-bundle $\Bl_{\{pt\}} \pr^2\to \pr^1$.
In the latter case, ${S_m} \to \pr^1$ is given by a semiample divisor $l-e$, where $l$ is the transform of a general line for some blow-up ${S_m} \to \pr^2$ and $e$ is one of the exceptional curves.
As in the proof of Proposition \ref{m=3}, we see that $H-E:=\alpha(l-e)$ is the supporting divisor of a fibration $\omega\colon X \to \pr^2$ that factors through a quadric bundle as in Lemma \ref{quadric-bundle}.
It follows that $\beta(F) \subseteq\NE(\omega)$ by Remark \ref{alpha}.
\end{proof}
\end{proof}

\begin{remark}\label{embedded surface}
The del Pezzo surface $S_m$ from Theorem \ref{m=8} (together with the isomorphism between the Mori cones and effective cones of $S_m$ and $Z_m$) is abstract, meaning that it does not need to be a subvariety of $Z_m$; nonetheless, this identification is possible in certain cases.

Let $\Pi_1,\dots,\Pi_{m+1}\subset Z$ be as in Setup \ref{setup}. Assume that $\Pi_1,\dots,\Pi_{m}\subset Z$ satisfy the generality conditions $\boldsymbol{(\star)}$, so that $Z_m$ is Fano, and that $\Pi_1,\dots,\Pi_{m+1}$ satisfy the generality condition \textbf{G1}.
We set $S_m:=P_{m+1}$ with inclusion $\iota\colon P_{m+1}\hookrightarrow Z_m$ and pushforward of one-cycles $\beta:=\iota_*$, so that every $(-1)$-curve $d$ of $P_{m+1}$ is the intersection $d=D\cap P_{m+1}\simeq f_D$ for some $D \in \Exc(Z_m)$, and the strict transform of a general line in $\Pi_{m+1}$ via $\mu_{m+1}\colon P_{m+1}\to\Pi_{m+1}$ is numerically equivalent to $h$.
See also \cite[Corollary 8.21]{Casa25}.

The identification above is possible for $m\leq 7$, but it is not clear whether $Z_8$ naturally contains a del Pezzo surface $\iota\colon S_8\hookrightarrow Z_8$ such that $\beta=\iota_*$.
However, this happens if, for instance, $Z$ contains a special configuration of planes; see Corollary \ref{co:base-locusV}.
\end{remark}

\begin{corollary}\label{Eff(X)}
Let $Z_m$ be as in Theorem \ref{m=8} with $m\geq 2$.
For every $$(d; k_1, k_2)\in\{(1;2,0),\, (2;5,0),\, (3;6,1),\, (4;5,3),\, (5;2,6),\, (6;7,1)\}$$ with $k:=k_1+k_2\leq m$, and for every partition $\mathfrak p_{(k_1,k_2)}:=\{a_1,\dots,a_{k_1}\mid b_1,\dots,b_{k_2}\}$ of $\{i_1, \dots, i_k\}\subseteq\{1, \dots, m\}$, there exists a unique irreducible hypersurface section
$$H_{d,\,\mathfrak p_{(k_1,k_2)}}\in|\ol_Z(d)|$$
containing $\Pi_j$ if and only if $j\in\{i_1, \dots, i_k\}$, and such that:
\begin{enumerate}[i.]\setlength\itemsep{4pt}
\item $H_{1,\,\{i_1, i_2\}}$ is smooth at the general point of $\Pi_{i_1},\Pi_{i_2}$;
\item $H_{2,\,\{i_1, \dots, i_5\}}$ is smooth at the general point of $\Pi_{i_1},\dots,\Pi_{i_5}$;
\item $H_{3,\,\mathfrak p_{(6,1)}}$ is smooth at the general point of $\Pi_{a_1},\dots,\Pi_{a_6}$ and has multiplicity two at the general point of $\Pi_{b_1}$;
\item $H_{4,\,\mathfrak p_{(5,3)}}$ is smooth at the general point of $\Pi_{a_1},\dots,\Pi_{a_5}$ and has multiplicity two at the general point of $\Pi_{b_1},\Pi_{b_2},\Pi_{b_3}$;
\item $H_{5,\,\mathfrak p_{(2,6)}}$ is smooth at the general point of $\Pi_{a_1},\Pi_{a_2}$ and has multiplicity two at the general point of $\Pi_{b_1},\dots,\Pi_{b_6}$;
\item $H_{6,\,\mathfrak p_{(7,1)}}$ has multiplicity two at the general point of $\Pi_{a_1},\dots,\Pi_{a_7}$ and has multiplicity three at the general point of $\Pi_{b_1}$.
\end{enumerate}
The strict transforms of the above hypersurface sections, together with $E_1, \dots, E_m$, form $\Exc(Z_m)$.
\end{corollary}
\begin{proof}
The existence and uniqueness of the hypersurface sections from items \emph{i.}\ to \emph{vi.}\ follow from the proof of Theorem \ref{m=8} and the description of $(-1)$-curves of smooth del Pezzo surfaces for $m\geq3$, and from $\S$\ref{z2} for $m=2$.
\end{proof}

\begin{corollary}\label{Nef(X)}
Let $Z_m$, $S_m$ and $\alpha$ be as in Theorem \ref{m=8}.
Then, the restriction of $\alpha$ to $\Nef(S_m)$ yields an isomorphism
$$\Nef(S_m)\xrightarrow{\simeq}\Nef(Z_m).$$
The extremal rays of $\Nef(Z_m)$ are generated by divisors of the form $H_\zeta$ or $H_\zeta-E_{\zeta,i}$ for some $\sigma_\zeta \colon Z_m \to Z_\zeta$ and $\zeta \in {\rm Cris}(Z_m)$.
\end{corollary}
\begin{proof}
It follows from $\S$\ref{z1} and $\S$\ref{z2} for $m\in\{1,2\}$, and from the proof of Theorem \ref{m=8}\emph{ii.} for $m\geq 3$.
\end{proof}

\begin{corollary}\label{contractions:X}
Let $Z_m$ be as in Theorem \ref{m=8}.
Then, every contraction of $Z_m$ is one of the following:
\begin{enumerate}[i.]\setlength\itemsep{4pt}
\item a smooth blow-up that is either $\sigma_\zeta \colon Z_m \to Z_\zeta$ for some $\zeta \in {\rm Cris}(Z_m)$, or one of its intermediate blow-ups $Z_m \to Z_{\zeta,I}$ as in \eqref{m!};
\item a fibration $\omega\colon  Z_m \to \pr^2$ in del Pezzo surfaces that factors through the quadric bundle $\ph\colon Z_1 \to \pr^2$ from Lemma \ref{quadric-bundle};
\item a smooth blow-up $\kappa\colon Z_m \to T$ onto a Verra fourfold $T$ that factors through the elementary blow-up $\psi\colon Z_2 \to T$ from Remark \ref{H_12}.
\end{enumerate}
In particular, $Z_m$ has no small contractions.
\end{corollary}
\begin{proof}
We know from Theorem \ref{m=8}\emph{ii.}\ that $\alpha$ induces a one-to-one correspondence between the contractions of $Z_m$ and the contractions of $S_m$, and the claim follows from Setup \ref{setup}, Lemma \ref{quadric-bundle} and Remark \ref{H_12}.
We have that item \emph{i.}\ corresponds to $S_m\to \pr^2$, item \emph{ii.}\ to $S_m\to \pr^1$ and item \emph{iii.}\ to $S_m\to \PP$.
\end{proof}

\begin{proof}[Proof of Theorem \ref{main:fano}]
It follows from Theorem \ref{m=8} and Corollary \ref{contractions:X}.
\end{proof}

\begin{proof}[Proof of Theorem \ref{main:cone}]
It follows from Theorem \ref{m=8} and Corollaries \ref{Eff(X)}, \ref{Nef(X)} and \ref{contractions:X}.
\end{proof}

\section{Existence and a special configuration of planes}\label{sec:existence}
In this section we prove Theorem \ref{existence}, that is the existence of a smooth cubic fourfold $Z\subset \pr^5$ and $m\leq 8$ planes $\Pi_1,\dots,\Pi_m\subset Z$ as in Setup \ref{setup} that satisfy the generality conditions $\boldsymbol{(\star)}$.

To do so, we consider a special configuration of planes $\Pi^{\ma V}_1, \dots, \Pi^{\ma V}_m\subset\pr^5$ satisfying the generality conditions \textbf{G0} and \textbf{G1}, that is the tangent planes to the Veronese surface $\ma V\subset \pr^5$, see $\S$\ref{veronese}.
In $\S$\ref{sec:horace} we show the existence of a smooth cubic fourfold containing such planes and satisfying the remaining generality condition \textbf{G2}: we use the residual exact sequence \eqref{horace} from $\S$\ref{ses}, a powerful tool that allows us to compute the dimension of the linear system of cubic hypersurfaces in $\pr^5$ that contain $\Pi^{\ma V}_1\cup \dots\cup \Pi^{\ma V}_m$.

Notation as in Section \ref{ssetup}.

\subsection{Tangent planes to the Veronese surface}\label{veronese}
Let $\nu\colon \pr^2 \hookrightarrow \pr^5$ be the $2$-Veronese embedding of $\pr^2$, that is $\pr^2=\pr\bigl(\C[y_0, y_1,y_2]_1\bigr)$, $\pr^5=\pr\bigl(\C[y_0, y_1,y_2]_2\bigr)$ and
$$\ma V:=\nu(\pr^2)=\{[L^2]\mid [L]\in\pr^2\}.$$
A classical result, usually credited to Lasker \cite{Lasker}, gives an explicit description of the tangent planes to $\ma V$ in terms of products of linear forms, that is
$$T_{\nu(L)}\ma V=\{[L\cdot L']\mid [L']\in \pr^2\}.$$
Thus, the tangent planes corresponding to distinct points $[L_1],[L_2]\in\pr^2$ intersect along the unique point $[L_1\cdot L_2]\in \pr^5$.
As a direct consequence we have the following.

\begin{lemma}\label{tangent-planes}
Let $\{\Pi^{\ma V}_i:=T_{\nu(L_i )}\ma V\}_{i\in I}\subset\pr^5$ be the tangent planes corresponding to points $\{[L_i]\}_{i\in I}\subset \pr^2$. Then, for all $i\in I$ there exists an isomorphism
$$\theta_i\colon \pr^2\to \Pi^{\ma V}_i$$
such that
$$\theta_i(L_i)=\nu(L_i)=[L_i^2]= (\Pi^{\ma V}_i\cap \ma V)_{red} \quad \text{and} \quad \theta_i(L_j)=[L_i\cdot L_j]=\Pi^{\ma V}_i\cap \Pi^{\ma V}_j \ \forall j\neq i.$$
\end{lemma}

Therefore, any closed relation on $\{[L_i]\}_{i\in I}\subset \pr^2$ induces, through $\theta_i$, the same closed relation on $\{[L_i\cdot L_j]\}_{j\in I}\subset \Pi^{\ma V}_i$ for all $i\in I$, and we have the following.

\begin{corollary}\label{existence planes V}
Let $[L_1], \dots, [L_m]\in \pr^2$ be $m\leq 8$ points in general position, i.e.\ the blow-up of $\pr^2$ at $[L_1]\cup \dots\cup [L_m]$ is a smooth del Pezzo surface.
Then, the corresponding tangent planes $\Pi^{\ma V}_1, \dots, \Pi^{\ma V}_m\subset \pr^5$ are as in Setup \ref{setup} and satisfy the generality conditions \textbf{G0} and \textbf{G1}.
\end{corollary}
\begin{proof}
In view of Lemma \ref{tangent-planes}, we only need to show that the generality condition \textbf{G0} holds for $m=3$.
If there exists an hyperplane $H$ containing $\Pi^{\ma V}_1\cup \Pi^{\ma V}_2\cup \Pi^{\ma V}_3$, then the curve $\nu^{-1}(H\cap \ma V)$ is a conic with double points at $[L_1],\,[L_2],\,[L_3]$, hence a double line. We conclude that the points are collinear, a contradiction.
\end{proof}

The Veronese surface $\ma V$ naturally provides us with a cubic fourfold containing any number of planes as in Setup \ref{setup}: sadly, such cubic is singular.

\begin{example}\label{secant}
The Veronese surface $\ma V$ is $2$-defective, meaning that the $2$-secant variety
$$\sigma_2(\ma V)=\overline{\bigcup_{\substack{q_i, q_j \in \ma V\\ i\neq j}}\overline{q_i\, q_j}}$$
does not have the expected dimension: $\sigma_2(\ma V)$ is an irreducible determinantal cubic fourfold, singular along $\ma V$ (\cite[$\S1$ and references therein]{BO08}), and it is not difficult to see that $\sigma_2(\ma V)$ coincides with the tangential variety
$${\rm Tan}(\ma V):=\bigcup_{q\in\ma V}T_{q}\ma V.$$
\end{example}

\subsection{Cubic fourfolds containing planes}\label{sec:horace}
With Corollary \ref{existence planes V} we show the existence of at least a configuration of $m\leq8$ planes $\Pi_1,\dots,\Pi_m$ as in Setup \ref{setup} that satisfy the generality conditions \textbf{G0} and \textbf{G1}.
In this subsection, we conclude the last step towards Theorem \ref{existence}: we prove the existence of a smooth cubic fourfold $Z$ that contains $\Pi_1\cup\dots\cup\Pi_m$ and that satisfies the generality condition \textbf{G2}.

Consider the ideal sheaves
$$
\mathfrak I_{m,d}:=\ma I_{\Pi_1\cup\dots\cup\Pi_m/\pr^5}\otimes\ol_{\pr^5}(d)\quad\text{and}\quad \w{\mathfrak I}_{m,d}:=\ma I_{\Pi_1\cup\dots\cup\Pi_m\cup\w\Pi/\pr^5}\otimes\ol_{\pr^5}(d),
$$
for $d\geq1$, $m\geq 3$ and any given $\w\Pi:=\w\Pi_{ijk}$. Our objective is to show that the corresponding linear systems of cubic fourfolds containing $\Pi_1\cup\dots\cup\Pi_m$ and $\Pi_1\cup\dots\cup\Pi_m\cup\w\Pi$, respectively, are such that
$$|H^0(\pr^5,\w{\mathfrak I}_{m,3})|\subsetneq |H^0(\pr^5,\mathfrak I_{m,3})|,$$
and that a general element of $|H^0(\pr^5,\mathfrak I_{m,3})|$ is smooth.

\begin{remark}\label{exp dim}
We compute the expected dimension of $H^0(\pr^5,\mathfrak I_{m,3})$ and $H^0(\pr^5,\w{\mathfrak I}_{m,3})$.
Since the restriction map $H^0(\pr^5,\ol_{\pr^5}(3)) \to H^0(\pr^2, \ol_{\pr^2}(3))$ is surjective, containing a fixed plane imposes exactly $h^0(\pr^2,\ol_{\pr^2}(3))= 10$ linear conditions on the complete linear system of cubic fourfolds.
By construction, $\Pi_1,\dots,\Pi_m$ intersect at distinct points, thus containing $\Pi_1$ induces $10$ conditions, containing $\Pi_2$ induces $9$ conditions, and so on. Therefore, the \emph{expected dimension} $\delta_m$ of $H^0(\pr^5, \mathfrak I_{m,3})$ is
$$
\delta_m=h^0(\pr^5,\ol _{\pr^5}(3))-\sum_{n=0}^{m-1}(10-n)=\frac{1}{2}(m^2-21m+112).
$$
Note that $\delta_m$ is decreasing for $m\leq 10$.
Moreover, a cubic fourfold that contains $\Pi_i,\,\Pi_j,\,\Pi_k$ intersects $\w\Pi_{ijk}$ at least along $\ell_i^{jk}\cup\ell_j^{ik}\cup\ell_k^{ij}$, and this corresponds to $9$ conditions.
Therefore, a cubic fourfold through $\Pi_i,\,\Pi_j,\,\Pi_k$ requires one further condition to contain $\w\Pi_{ijk}$, i.e.\ the \emph{expected dimension} $\w \delta_m$ of $H^0(\pr^5, \w{\mathfrak I}_{m,3})$ is
$$
\w \delta_m=\delta_m-1=\frac{1}{2}(m^2-21m+110).
$$
In conclusion, we have $\delta_m \leq h^0(\mathfrak I_{m,3})$, $\w \delta_m\leq h^0(\w{\mathfrak I}_{m,3})$, and $h^0(\mathfrak I_{m,3})-1\leq h^0(\w{\mathfrak I}_{m,3}) \leq h^0(\mathfrak I_{m,3})$.
\end{remark}

\medskip

Now, let us take $m\leq 8$ planes $\Pi^{\ma V}_1, \dots, \Pi^{\ma V}_m\subset \pr^5$ tangent to the Veronese surface $\ma V\subset \pr^5$ as in Corollary \ref{existence planes V}.
In the following proposition we show that, at least for this configuration of planes, the expected dimensions $\delta_m$ and $\w \delta_m$ are in fact exact.

\begin{proposition}\label{expected-dim}
Let $\Pi^{\ma V}_1,\dots,\Pi^{\ma V}_m\subset\pr^5$ be $m\leq8$ planes as in Corollary \ref{existence planes V}. Then,
$$h^0(\pr^5, \mathfrak I_{m,3})=\delta_m  \quad \text{and} \quad h^0(\pr^5, \w{\mathfrak I}_{m,3})=\w \delta_m.$$
\end{proposition}

To prove Proposition \ref{expected-dim}, we first establish some preliminary results, namely Lemmas \ref{lem:planes-tildaV}, \ref{I_4(1)} and \ref{I_6(2)}.
For simplicity, we omit the superscript $\text{-}^{\ma V}$ in the remainder of this subsection. 
Recall the short exact sequences in $\S$\ref{ses} and the notation from $\S$\ref{geom-conf}.

\begin{lemma}\label{lem:planes-tildaV}
Let $\Pi_1,\dots,\Pi_6\subset\pr^5$ be planes as in Setup \ref{setup}, satisfying the generality conditions \textbf{G0} and \textbf{G1}.
Then:
\begin{enumerate}[i.]\setlength\itemsep{4pt}
\item $\w\Pi_{123}\cap\w\Pi_{456}=\emptyset$ if and only if  $\w\Pi_{ijk}\cap\w\Pi_{abc}=\emptyset$ for all partitions $\{i,j,k\mid a,b,c\}$ of $\{1,\dots,6\}$;
\item if $\Pi_1,\dots,\Pi_6$ are as in Corollary \ref{existence planes V}, then $\w\Pi_{ijk}\cap\w\Pi_{abc}=\emptyset$ for all partitions $\{i,j,k\mid a,b,c\}$ of $\{1,\dots,6\}$.
\end{enumerate}
\end{lemma}
\begin{proof}
We show that $\w\Pi_{123}\cap\w\Pi_{456}\neq\emptyset$ implies $\w\Pi_{126}\cap\w\Pi_{345}\neq\emptyset$. Since the argument does not depend on the choice of indices, we deduce that $\w\Pi_{ijk}\cap\w\Pi_{abc}\neq\emptyset$ for every partitions $\{i,j,k\mid a,b,c\}$ of $\{1,\dots,6\}$, hence we obtain \emph{i.}

Assume $\w\Pi_{123}\cap\w\Pi_{456}\neq\emptyset$.
We have
$$\w\Pi_{123}\cap\w\Pi_{456}=\w\Pi_{123}\cap(D_{45}\cap D_{46}\cap D_{56})=\w\ell_{123}^{45}\cap\w\ell_{123}^{46}\cap\w\ell_{123}^{56},$$
so that $\w\Pi_{123}\cap\w\Pi_{456}$ is either a point or a line.
In either case, $\w\Pi_{123}\cap\w\Pi_{456}\cap \ell_i^{jk}$ must be empty for $\{i,j,k\}=\{1,2,3\}$, otherwise the points $\{p_{i4},\,p_{i5},\,p_{i6}\}\subset \Pi_i$ would be collinear, a contradiction to the generality condition \textbf{G1}.
Therefore, we obtain that
$$\w\Pi_{123}\cap\w\Pi_{456}=\w\ell_{123}^{45}\cap\w\ell_{123}^{46}\cap\w\ell_{123}^{56}=\w\ell_{456}^{12}\cap\w\ell_{456}^{13}\cap\w\ell_{456}^{23}$$
is a point which does not lie on any line $\ell_u^{vw}$ with $\{u,v,w\}=\{1,2,3\},\{4,5,6\}$.
Since $\w\Pi_{123}$ is disjoint from $\Pi_i$ for all $i\in\{4,5,6\}$ by Lemma \ref{4planes}, we have that the lines $\ell_4^{12},\, \ell_5^{12},\,\w\ell_{123}^{45}$ are pairwise disjoint: let $Q\cong\PP$ be the unique quadric surface in $D_{12}\cap D_{45}\cong\pr^3$ containing them (see Figure \ref{lines-rulingsV}).
Now, let $r_3$ and $r_6$ be the points $D_{12}\cap D_{45} \cap \Pi_3$ and $D_{12}\cap D_{45} \cap \Pi_6$, respectively: we have $r_3=\ell_3^{12}\cap\ell_3^{45}$, thus $r_3=\w\ell_{123}^{45}\cap\w\ell_{345}^{12}$, and similarly $r_6=\w\ell_{126}^{45}\cap\w\ell_{456}^{12}$.
The line $\w\ell_{345}^{12}$ intersects each line $\ell_4^{12},\, \ell_5^{12},\,\w\ell_{123}^{45}$ at a point, so $\w\ell_{345}^{12}\subset Q$.
Since $\w\Pi_{123}\cap\w\Pi_{456}\neq\emptyset$, we have that the lines $\w\ell_{123}^{45}$ and $\w\ell_{456}^{12}$ intersect, so $\w\ell_{456}^{12}\subset Q$ too.
Finally, this implies $\w\ell_{126}^{45}\subset Q$ and so $\w\ell_{126}^{45}\cap \w\ell_{345}^{12}\neq \emptyset$, that is $\w\Pi_{126}\cap\w\Pi_{345}\neq\emptyset$.
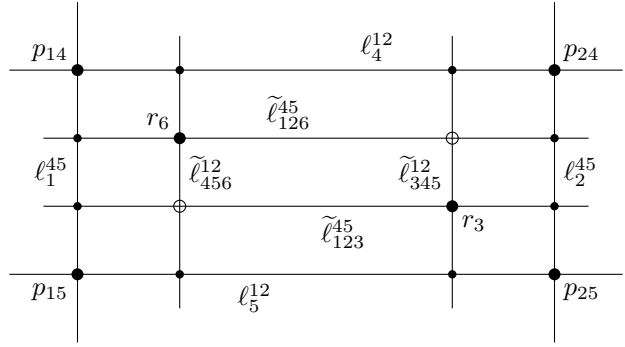
\begin{figure}[h]\caption{\small Lines in $D_{12}\cap D_{45}$.}\label{lines-rulingsV}
\vspace{10pt}
\begin{tikzpicture}[scale=.9]
\draw (0,-1) -- (0,4)
      node[midway,left] {$\ell_1^{45}$};

\draw (1.5,-0.5) -- (1.5,3.5)
      node[midway,right] {$\w\ell_{456}^{12}$};

\draw (5.5,-0.5) -- (5.5,3.5)
      node[midway,left] {$\w\ell_{345}^{12}$};

\draw (7,-1) -- (7,4)
      node[midway,right] {$\ell_2^{45}$};

\draw (-1,0) -- (8,0)
      node[pos=.4,below] {$\ell_5^{12}$};

\draw (-0.5,1) -- (7.5,1)
      node[pos=.55,below] {$\w\ell_{123}^{45}$};

\draw (-0.5,2) -- (7.5,2)
      node[pos=.45,above] {$\w\ell_{126}^{45}$};

\draw (-1,3) -- (8,3)
      node[pos=.6,above] {$\ell_4^{12}$};

\fill (0,3) circle (2.5pt)
      node[above left] {$p_{14}$};

\fill (0,0) circle (2.5pt)
      node[below left] {$p_{15}$};

\fill (1.5,2) circle (2.5pt)
      node[above left] {$r_6$};

\fill (5.5,1) circle (2.5pt)
      node[below right] {$r_3$};

\fill (7,3) circle (2.5pt)
      node[above right] {$p_{24}$};

\fill (7,0) circle (2.5pt)
      node[below right] {$p_{25}$};

\draw (1.5,1) circle (2.5pt);

\draw (5.5,2) circle (2.5pt);

\fill (0,1) circle (1.75pt);
\fill (0,2) circle (1.75pt);
\fill (1.5,0) circle (1.75pt);
\fill (5.5,0) circle (1.75pt);
\fill (1.5,3) circle (1.75pt);
\fill (5.5,3) circle (1.75pt);
\fill (7,1) circle (1.75pt);
\fill (7,2) circle (1.75pt);
\end{tikzpicture}
\end{figure}

Let now assume that $\Pi_1,\dots,\Pi_6$ are as in Corollary \ref{existence planes V}.
That is, the points
$$[L_1], \dots, [L_6]\in \pr^2=\pr\bigl(\C[y_0, y_1,y_2]_1\bigr)$$
corresponding to $\Pi_1,\dots,\Pi_6$ do not lie on a conic, and $p_{ij}=[L_i\cdot L_j]\in \pr^5$ (cf.\ $\S$\ref{veronese}).

To prove \emph{ii.}, by \emph{i.}\ it suffices to show that $\w\Pi_{123}\cap\w\Pi_{456}=\emptyset$.
Assume by contradiction that $\w\Pi_{123}\cap\w\Pi_{456}\neq\emptyset$.
Then, as above, we have that $\w\Pi_{123}\cap\w\Pi_{456}$ is a point not lying on any line $\ell_u^{vw}$ with $\{u,v,w\}=\{1,2,3\},\{4,5,6\}$.
Thus, we have is a linear combination
\begin{equation}\label{lin.comb.}
c_{12}L_1 L_2+c_{13}L_1L_3+c_{23}L_2 L_3=c_{45}L_4L_5+c_{46}L_4L_6+c_{56}L_5 L_6
\end{equation}
with non-vanishing coefficients $\{c_{ij}\}$.
In the dual projective plane $(\pr^2)^*$ we have that the lines $\{L_1, L_2, L_3\}$ and $\{L_4, L_5, L_6\}$ are not concurrent, hence \eqref{lin.comb.} gives a smooth conic through the vertices of the two triangles given by the two triples of lines.
By a classical result, we know that there exists a unique smooth conic $C\subset (\pr^2)^*$ that is inscribed in the two tringles, i.e.\ $C$ is tangent to $L_i$ for all $i\in\{1, \dots, 6\}$.
Therefore, the dual conic $C^*\subset \pr^2$ is smooth and contains $[L_i]$ for all $i\in\{1, \dots, 6\}$, a contradiction.
\end{proof}

\begin{lemma}\label{I_4(1)}
Let $\Pi_1,\dots,\Pi_4\subset\pr^5$ be planes as in Corollary \ref{existence planes V}. Then, $h^i(\pr^5,\mathfrak I_{4,1})=0$ for $i\in\{0,1,2\}$.
\end{lemma}
\begin{proof}
We clearly have $h^0(\mathfrak I_{4,1})=0$ by the generality condition \textbf{G0}, so by the exact sequence
$$
0\to \mathfrak I_{4,1} \to \ol_{\pr^5}(1) \to \ol_{\Pi_1\cup\dots\cup\Pi_4}(1)\to 0
$$
it remains to show that $h^0(\ol_{\Pi_1\cup\dots\cup\Pi_4}(1))=6$ and $h^1(\ol_{\Pi_1\cup\dots\cup\Pi_4}(1))=0$.
From \eqref{union} we get
$$
0\to \ma I_{n-1}:=\ma I_{\Pi_n\cap(\Pi_1 \cup \dots \cup \Pi_{n-1})/ \Pi_n}\otimes \ol_{\Pi_n}(1)\to \ol_{\Pi_1\cup\dots\cup\Pi_n}(1) \to \ol_{\Pi_1\cup\dots\cup\Pi_{n-1}}(1)\to 0,
$$
where $\ma I_{n-1}$ is isomorphic to the ideal sheaf of $n-1$ points in general position in $\pr^2$ twisted by $\ol_{\pr^2}(1)$, so $h^1(\ma I_{n-1})=h^2(\ma I_{n-1})=0$ for $2\leq n\leq 4$. 
Therefore,
\begin{equation}\label{cohomology planes}
\begin{gathered}
h^0(\ol_{\Pi_1\cup\dots\cup\Pi_n}(1))=h^0(\ol_{\Pi_1\cup\dots\cup\Pi_{n-1}}(1))+h^0(\ma I_{n-1})\quad \text{and} \\
h^1(\ol_{\Pi_1\cup\dots\cup\Pi_n}(1))=h^1(\ol_{\Pi_1\cup\dots\cup\Pi_{n-1}}(1))
\end{gathered}
\end{equation}
for $2 \leq n\leq 4$, and we conclude that $h^0(\ol_{\Pi_1\cup\dots\cup\Pi_4}(1))=6$ and $h^1(\ol_{\Pi_1\cup\dots\cup\Pi_4}(1))=0$.
\end{proof}

\begin{lemma}\label{I_6(2)}
Let $\Pi_1,\dots,\Pi_6\subset\pr^5$ be planes as in Corollary \ref{existence planes V}. Then, $h^i(\pr^5,\mathfrak I_{6,2})=0$ for $i\in\{0,1\}$.
\end{lemma}
\begin{proof}
We apply \eqref{horace} to $W=\pr^5$, $X=\Pi_1\cup\dots\cup\Pi_6$ and $D=D_{12}$ to obtain the exact sequence
$$
0\to \mathfrak I_{4,1}\to \mathfrak I_{6,2} \to \ma F:= \ma I_{\Pi_1\cup\Pi_2 \cup\ell_3^{12}\cup\dots\cup\ell_6^{12}/\pr^4}\otimes\ol_{\pr^4}(2)\to 0,
$$
and we deduce from Lemma \ref{I_4(1)} that $h^i(\mathfrak I_{6,2})=h^i(\ma F)$ for $i\in\{0,1\}$.
Applying again \eqref{horace} to $W=D_{12}$, $X=\Pi_1\cup\Pi_2 \cup\ell_3^{12}\cup\dots\cup\ell_6^{12}$ and $D=D_{12}\cap D_{45}$ gives
$$
0\to \ma H:=\ma I_{\Pi_1\cup\Pi_2 \cup\ell_3^{12}\cup\ell_6^{12}/\pr^4}\otimes\ol_{\pr^4}(1)\to \ma F \to \ma G:= \ma I_{L\cup\{r_3,\,r_6\}/\pr^3}\otimes\ol_{\pr^3}(2)\to 0,
$$
where $L:=\ell_1^{45}\cup\ell_2^{45} \cup\ell_4^{12}\cup\ell_5^{12}$, $r_3:=D_{12}\cap D_{45} \cap \Pi_3$ and $r_6:=D_{12}\cap D_{45} \cap \Pi_6$.

\begin{step}\label{G I_6}
We have $h^i(\ma G)=0$ for all $i\geq0$.
\end{step}
\begin{proof}
We argue as in the the proof of Lemma \ref{lem:planes-tildaV}\emph{i.}; see Figure \ref{lines-rulingsV}.
Let $Q\cong\PP$ be the unique quadric surface in $D_{12}\cap D_{45}$ containing $\ell_4^{12},\, \ell_5^{12},\,\w\ell_{123}^{45}$: since $\ell_1^{45}$ and $\ell_2^{45}$ intersect $Q$ at least in three points, it follows that $Q$ contains both $\ell_1^{45}$ and $\ell_2^{45}$. Moreover, $r_3\in Q$ since it belongs to $\w\ell_{123}^{45}$, while $r_6\not\in Q$, otherwise $\w\Pi_{123}\cap\w\Pi_{456}\neq\emptyset$ and we get a contradiction to Lemma \ref{lem:planes-tildaV}\emph{ii.}
We apply \eqref{horace} to $W=D_{12}\cap D_{45}$, $X=L\cup\{r_3,\,r_6\}$ and $D=Q$ and get
$$
0\to \ma I_{\{r_6\}/\pr^3}\to \ma G \to \ma G':=\ma I_{L\cup\{r_3\}/\PP}\otimes\ol_{\PP}(2,2)\to 0.
$$
Note that $L\simeq\ol_{\PP}(2,2)$ as a divisor in $Q$: since $L$ is Cartier and $r_3\not\in L$, we have $\ma I_{L\cup\{r_3\}/\PP} \cong \ma I_{\{r_3\}/\PP}\otimes\ol_{\PP}(-2,-2)$.
Both $\ma I_{\{r_6\}/\pr^3}$ and $\ma G'\cong\ma I_{\{r_3\}/\PP}$ have vanishing cohomology, thus the same holds for $\ma G$.
\end{proof}

\begin{step}\label{H I_6}
We have $h^0(\ma H)=h^1(\ma H)=0$.
\end{step}
\begin{proof}
An hyperplane in $\pr^4$ does not contain two planes meeting along a point, hence $h^0(\ma H)=0$.
From \eqref{union} we get
$$
0\to \ma I_{\{p_{16},\,p_{26}\}/ \ell_6^{12}}\otimes \ol_{\pr^1}(1)\cong\ol_{\pr^1}(-1)\to \ol_{\Pi_1\cup\Pi_2\cup\ell_3^{12}\cup\ell_6^{12}}(1) \to \ol_{\Pi_1\cup\Pi_2\cup\ell_3^{12}}(1)\to 0
$$
since $\ell_3^{12}\cap \ell_6^{12}=\emptyset$ and $(\Pi_1\cup\Pi_2\cup\ell_3^{12})\cap\ell_6^{12}=\{p_{16},\,p_{26}\}$, and similarly we get
$$
0\to \ol_{\pr^1}(-1)\to \ol_{\Pi_1\cup\Pi_2\cup\ell_3^{12}}(1) \to \ol_{\Pi_1\cup\Pi_2}(1)\to 0.
$$
We deduce that $h^0(\ol_{\Pi_1\cup\Pi_2\cup\ell_3^{12}\cup\ell_6^{12}}(1))=5$ from \eqref{cohomology planes} and the short exact sequences above, and so we get $h^1(\ma H)=0$ from
$$
0\to \ma H \to \ol_{\pr^4}(1)\to \ol_{\Pi_1\cup\Pi_2\cup\ell_3^{12}\cup\ell_6^{12}}(1)\to 0.
$$
\end{proof}
Finally, we deduce from Steps \ref{G I_6} and \ref{H I_6} that $h^0(\ma F)=h^1(\ma F)=0$, which concludes the proof.
\end{proof}

\begin{proof}[Proof of Proposition \ref{expected-dim}]
By Remark \ref{exp dim}, it suffices to show $h^0(\pr^5, \w{\mathfrak I}_{8,3})=\w \delta_8=3$ for $\w\Pi:=\w\Pi_{123}$. We apply \eqref{horace} to $W=\pr^5$, $X=\Pi_1\cup\dots\cup\Pi_8\cup\w\Pi$ and $D=D_{12}$ to obtain the exact sequence
$$
0\to \mathfrak I_{6,2}\to \w{\mathfrak I}_{8,3}\to \ma F:= \ma I_{\Pi_1\cup\Pi_2\cup\w\Pi \cup\ell_4^{12}\cup\dots\cup\ell_8^{12}/\pr^4}\otimes\ol_{\pr^4}(3)\to 0,
$$
and we deduce from Lemma \ref{I_6(2)} that $h^0(\w{\mathfrak I}_{8,3})=h^0(\ma F)$.
Applying again \eqref{horace} to $W=D_{12}$, $X=\Pi_1\cup\Pi_2\cup\w\Pi \cup\ell_4^{12}\cup\dots\cup\ell_8^{12}$ and $D=D_{12}\cap D_{45}$ gives
$$
0\to \ma H:=\ma I_{\Pi_1\cup\Pi_2\cup\w\Pi \cup\ell_6^{12}\cup\ell_7^{12}\cup\ell_8^{12}/\pr^4}\otimes\ol_{\pr^4}(2)\to \ma F \to \ma G:= \ma I_{L\cup\{r_6,\,r_7,\,r_8\}/\pr^3}\otimes\ol_{\pr^3}(3)\to 0,
$$
where $L:=\ell_1^{45}\cup\ell_2^{45} \cup \w\ell_{123}^{45}\cup\ell_4^{12}\cup\ell_5^{12}$ and $r_i:=D_{12}\cap D_{45} \cap \Pi_i$ for $i\in\{6,7,8\}$.

\begin{step}\label{G I_8}
We have $h^0(\ma G)=3$.
\end{step}
\begin{proof}
We argue as in the the proof of Lemma \ref{lem:planes-tildaV}\emph{i.}; see Figure \ref{lines-rulingsV}.
Similarly to Step \ref{G I_6}, let $Q\cong \PP$ be the unique quadric surface in $D_{12}\cap D_{45}$ containing $L$, so $r_6,r_7,r_8\not\in Q$ by Lemma \ref{lem:planes-tildaV}\emph{ii.}
We apply \eqref{horace} to $W=D_{12}\cap D_{45}$, $X=L\cup\{r_6,\,r_7,\,r_8\}$ and $D=Q$, and we obtain
$$
0\to \ma G'':=\ma I_{\{r_6,\,r_7,\,r_8\}/\pr^3}\otimes\ol_{\pr^3}(1)\to \ma G \to \ma I_{L/\PP}\otimes\ol_{\PP}(3,3)\cong\ol_{\PP}(0,1)\to 0,
$$
since $L\simeq\ol_{\PP}(3,2)$ as a divisor in $Q$.
We show that we can assume $\{r_6,\,r_7,\,r_8\}$ to be not collinear; the claim then follows from $h^0(\ma G'')=1$, $h^1(\ma G'')=0$ and $h^0(\ol_{\PP}(0,1))=2$.

Assume by contradiction that $\{r_6,\,r_7,\,r_8\}$ lie on a line $\ell$.
By the generality condition \textbf{G1} the points $\{r_6,\,r_7,\,r_8\}$ are pairwise
distinct, hence $\ell\subset\w\Pi_{678}$ by Lemma \ref{lemma:geometry-3planes},
and $r_6\in\ell_6^{78}$.
Thus, the lines $\ell_6^{12},\,\ell_6^{45},\,\ell_6^{78}$ are concurrent at $r_6$, and this implies that the points $s_i:=D_{13}\cap D_{45} \cap \Pi_i$ for $i\in\{6,7,8\}$ are not collinear. Otherwise, the lines $\ell_6^{13},\,\ell_6^{45},\,\ell_6^{78}$ would be concurrent at $s_6$, so that $r_6=s_6$.
It would then follow that $\{p_{16},\,p_{26},\,p_{36}\}$ are collinear, a contradiction to the generality condition \textbf{G1}.
We conclude that we can repeat the proof for $D_{13}\cap D_{45}$ instead of $D_{12}\cap D_{45}$, and the claim follows.
\end{proof}

\begin{step}\label{H I_8}
We have $h^0(\ma H)=h^1(\ma H)=0$.
\end{step}
\begin{proof}
We apply \eqref{horace} to $W=D_{12}$, $X=\Pi_1\cup\Pi_2\cup\w\Pi \cup\ell_6^{12}\cup\ell_7^{12}\cup\ell_8^{12}$ and $D=D_{12}\cap D_{67}$, and we get
$$
0\to \ma H'':=\ma I_{\Pi_1\cup\Pi_2\cup\w\Pi \cup\ell_8^{12}/\pr^4}\otimes\ol_{\pr^4}(1)\to \ma H \to \ma H':= \ma I_{L'\cup\{r_8'\}/\pr^3}\otimes\ol_{\pr^3}(2)\to 0
$$
where $L':=\ell_1^{67}\cup\ell_2^{67} \cup \w\ell_{123}^{67}\cup\ell_6^{12}\cup\ell_7^{12}$ and $r_8':=D_{12}\cap D_{67} \cap \Pi_8$.
Similarly to Step \ref{G I_6}, we see that $h^0(\ma H')=h^1(\ma H')=0$: $r_8'$ does not belong to the unique quadric surface $Q\cong\PP$ in $D_{12}\cap D_{67}$ containing $L'$, and we apply again \eqref{horace} to $W=D_{12}\cap D_{67}$, $X=L'\cup\{r_8'\}$ and $D=Q$.
Therefore, it remains to show that $h^0(\ma H'')=h^1(\ma H'')=0$.
As in Step \ref{H I_6} we have $h^0(\ma H'')=0$, $h^0(\ol_{\Pi_1\cup\Pi_2\cup\ell_8^{12}}(1))=5$, and we have
$$
0\to \ma I_{\ell_1^{23}\cup\ell_2^{13}/\w\Pi}\otimes \ol_{\w\Pi}(1)\cong\ol_{\pr^2}(-1)\to \ol_{\Pi_1\cup\Pi_2\cup\w\Pi\cup\ell_8^{12}}(1) \to \ol_{\Pi_1\cup\Pi_2\cup\ell_8^{12}}(1)\to 0
$$
from \eqref{union} and Lemma \ref{4planes}.
We conclude that $h^0(\ol_{\Pi_1\cup\Pi_2\cup\w\Pi\cup\ell_8^{12}}(1))=5$, hence $h^1(\ma H'')=0$ by
$$
0\to \ma H'' \to \ol_{\pr^4}(1)\to \ol_{\Pi_1\cup\Pi_2\cup\w\Pi\cup\ell_8^{12}}(1)\to 0.
$$
\end{proof}

Finally, we deduce from Steps \ref{G I_8} and \ref{H I_8} that $h^0(\ma F)=3$, and this concludes the proof.
\end{proof}

\begin{remark}\label{comment-proof}
In order to prove Proposition \ref{expected-dim}, we use the configuration of planes from Corollary \ref{existence planes V} only to show Lemma \ref{lem:planes-tildaV}\emph{ii.}

If we assume that $\w\Pi_{ijk}\cap\w\Pi_{abc}\neq\emptyset$ for some partition $\{i,j,k\mid a,b,c\}$ of $\{1,\dots,6\}$, as in the proof of Lemma \ref{lem:planes-tildaV} we see that we have a linear combination with non-vanishing coefficients
\begin{equation}\label{combinationsV}
c_{ij}^{\mathfrak p}p_{ij}+c_{ik}^{\mathfrak p}p_{ik}+c_{jk}^{\mathfrak p}p_{jk}=c_{ab}^{\mathfrak p}p_{ab}+c_{ac}^{\mathfrak p}p_{ac}+c_{bc}^{\mathfrak p}p_{bc}
\end{equation}
for every partition $\mathfrak p:=\{i,j,k\mid a,b,c\}$ of $\{1,\dots,6\}$, where we consider $p_{ij}$ as a fixed affine point in the corresponding line in $\C^6$.
Each pair of points $(p_{ij},\,p_{ab})$ appears precisely in two combinations as in \eqref{combinationsV}, that is the ones corresponding to the partitions $\mathfrak p=\{i,j,k\mid a,b,c\}$ and $\mathfrak p'=\{i,j,c\mid a,b,k\}$, and we have $c_{ij}^{\mathfrak p}\cdot c_{ab}^{\mathfrak p'}=c_{ab}^{\mathfrak p}\cdot c_{ij}^{\mathfrak p'}$.
Otherwise, it follows from \eqref{combinationsV} that
$$
\left(c_{ij}^{\mathfrak p}\cdot c_{ab}^{\mathfrak p'}- c_{ab}^{\mathfrak p}\cdot c_{ij}^{\mathfrak p'}\right)p_{ab}=c_{ik}^{\mathfrak p}p_{ik}+c_{jk}^{\mathfrak p}p_{jk}+c_{ak}^{\mathfrak p'}p_{ak}+c_{bk}^{\mathfrak p'}p_{bk}-c_{ic}^{\mathfrak p'}p_{ic}-c_{jc}^{\mathfrak p'}p_{jc}-c_{ac}^{\mathfrak p}p_{ac}-c_{bc}^{\mathfrak p}p_{bc}
,$$
that is $p_{ab}\in \Pi_k+\Pi_c=D_{kc}$, a contradiction to the generality condition \textbf{G1}.
Since this holds for all pairs of points $(p_{ij},\,p_{ab})$, it follows that there exist \virg{global} coefficients $\{\gamma_{ij}\}_{1\leq i<j\leq 6}$ and linear combinations
\begin{equation}\label{combinations2V}
\gamma_{ij}p_{ij}+\gamma_{ik}p_{ik}+\gamma_{jk}p_{jk}=\gamma_{ab}p_{ab}+\gamma_{ac}p_{ac}+\gamma_{bc}p_{bc}
\end{equation}
such that \eqref{combinations2V} coincides with \eqref{combinationsV} up to multiplication with some coefficient $\lambda^{\mathfrak p}\neq 0$ that only depends on the partition $\mathfrak p$.
By combining \eqref{combinations2V} for partitions $\{i,j,k\mid a,b,c\}$, $\{i,j,c\mid a,b,k\}$ and $\{i,k,c\mid a,b,j\}$ we obtain
\begin{equation}\label{combinations3V}
\gamma_{ik}p_{ik}+\gamma_{jk}p_{jk}+\gamma_{ak}p_{ak}+\gamma_{bk}p_{bk}+\gamma_{ck}p_{ck}=0
\end{equation}
for all $1\leq k \leq 6$.
Note that equations \eqref{combinations3V} hold for any configuration of $6$ planes as in Setup \ref{setup}: these conditions seem quite restrictive and probably do not hold for the general choice of planes, but we do not know the geometry of the corresponding moduli space.
We expect Lemma \ref{lem:planes-tildaV}\emph{ii.}\ to hold for the general configuration of planes as in Setup \ref{setup} and not limitedly to the one from Corollary \ref{existence planes V}.
\end{remark}

\begin{proof}[Proof of Theorem \ref{existence}]
We show that there exists a smooth cubic fourfold $Z$ and $m\leq 8$ planes $\Pi_1,\dots,\Pi_m\subset Z$ as in Setup \ref{setup} that satisfy the generality conditions $\boldsymbol{(\star)}$, namely the generality conditions \textbf{G0}, \textbf{G1} and \textbf{G2} (see $\S$\ref{gen_cond}).

By Corollary \ref{existence planes V} there exists a configuration of planes $\Pi_1,\dots,\Pi_m$ as in Setup \ref{setup} that satisfies \textbf{G0} and \textbf{G1}, while Proposition \ref{expected-dim} implies that a general cubic fourfold $Z$ containing $\Pi_1,\dots,\Pi_m$ satisfies \textbf{G2}.
Finally, the smoothness of $Z$ follows from Lemma \ref{appendix}.
\end{proof}

\begin{remark}\label{anticanonical sections}
As a consequence of Theorems \ref{main:fano} and \ref{existence}, we have that $h^0(\pr^5, \mathfrak I_{m,3})=\delta_m$ for all $\Pi_1,\dots,\Pi_m\subset Z$ as in Setup \ref{setup} that satisfy the generality conditions $\boldsymbol{(\star)}$.

In fact, let $Z_m$ be the blow-up of $Z$ along $\Pi_1\cup\dots\cup\Pi_m$. The projection formula and \eqref{anticanonical} give
\begin{equation}\label{pushforward}
\sigma_*\ol_{Z_m}(-K_{Z_m})\cong \sigma_*\ol_{Z_m}(-\Exc(\sigma))\otimes \ol_Z(-K_Z) \cong\ma I_{\Pi_1\cup\dots\cup\Pi_m/Z}\otimes\ol_{Z}(3).
\end{equation}
Moreover, \eqref{horace} applied to $W=\pr^5$, $X=\Pi_1\cup\dots\cup\Pi_m$ and $D=Z$ yields
\begin{equation}\label{eq:restriction base locus}
0 \to \ol_{\pr^5}\to \mathfrak I_{m,3} \to \ma I_{\Pi_1\cup\dots\cup\Pi_m/Z}\otimes\ol_{Z}(3) \to 0.
\end{equation}
Therefore, we have isomorphisms
$$
H^0(Z_m, -K_{Z_m})\cong H^0\left(Z, \ma I_{\Pi_1\cup\dots\cup\Pi_m/Z}\otimes\ol_{Z}(3)\right)\cong H^0(\pr^5, \mathfrak I_{m,3}) \big/ \langle s_Z\rangle
$$
where $s_Z\in H^0(\pr^5, \mathfrak I_{m,3})$ is the section corresponding to $Z$.  
Finally, we deduce from Table \ref{final table} that $h^0(Z_m, -K_{Z_m})=\delta_m-1$, where $h^0(Z_m, -K_{Z_m})$ is computed independently (see $\S$\ref{num.inv.}), and the claim follows.
\end{remark}

\section{Anticanonical base locus}\label{abl}
In this section we study the anticanonical linear system $|-K_{Z_m}|$ of the Fano fourfold $Z_m$ from Theorem \ref{main:fano} and we prove Theorem \ref{thm:base-locus}.

Moreover, with Corollary \ref{co:base-locusV} we strengthen the correspondence between $\sigma\colon Z_m\to Z$ and and the blow-up $g\colon S_m \to \pr^2$ along $m$ points in general position, already established in $\S$\ref{8}, in the special case of planes $\Pi^{\ma V}_1, \dots, \Pi^{\ma V}_m\subset Z$ tangent to a Veronese surface $\ma V\subset \pr^5$.
In particular, we show that there exists a ninth plane $\Pi^{\ma V}_9$ that belongs to the base locus of cubic hypersurface sections of $Z$ through $\Pi^{\ma V}_1, \dots, \Pi^{\ma V}_8$, parallel to the ninth point in the base locus of plane cubics through $8$ points in general position.

Notation as in Sections \ref{ssetup} and \ref{sec:existence}.

\begin{proposition}\label{prop:base-locusV}
Let $\Pi_1^{\ma V},\dots,\Pi_8^{\ma V}\subset\pr^5$ be planes as in Corollary \ref{existence planes V}. There exists a ninth plane $\Pi_9^{\ma V}\subset \pr^5$ tangent to $\ma V$ such that:
\begin{enumerate}[i.]\setlength\itemsep{4pt}
\item any eight planes in $\{\Pi_1^{\ma V},\dots,\Pi_9^{\ma V}\}$ are as in Corollary \ref{existence planes V};
\item the base locus $\Bs(|H^0(\mathfrak I_{8,3})|)$ contains $\Pi_1^{\ma V}\cup\dots\cup\Pi_9^{\ma V}$.
\end{enumerate}
\end{proposition}
First, we prove the following lemma.
\begin{lemma}\label{sextics}
Let $q_1,\dots,q_8\in\pr^2$ be points in general position, i.e.\ the blow-up of $\pr^2$ at $\{q_1, \dots,q_8\}$ is a smooth del Pezzo surface, and set $\mathfrak d:=|H^0(\ma I_{\{q_1,\dots,q_8\}/\pr^2}\otimes\ol_{\pr^2}(3))|$.
Let $q_9\in \pr^2$ be the unique point such that $\Bs(\mathfrak d)=\{q_1, \dots, q_9\}$.
Then,
$$|H^0(\ma I_{\{q_1^2,\dots,q_8^2,q_9\}/\pr^2}\otimes\ol_{\pr^2}(6))|=|\{C_1\cup C_2\mid C_i\in \mathfrak d\}|\cong\pr^2.$$
\end{lemma}
\begin{proof}
We have $\mathfrak d\cong\pr^1$, so
$$|\{C_1\cup C_2\mid C_i\in \mathfrak d\}|\cong \PP/\iota\cong\pr^2,$$
where $\iota$ is the involution.
By the Alexander-Hirschowitz Theorem \cite[Theorem 2]{AH95}\footnote{The good postulation of double points in the plane was already proved by Castelnuovo \cite{Castelnuovo} with different methods. See also \cite[Theorem 1.1]{BO08} for a modern formulation.}, we have that $h^0(\ma I_{\{q_1^2,\dots,q_8^2\}/\pr^2}\otimes\ol_{\pr^2}(6))=4$.
Since
$$|\{C_1\cup C_2\mid C_i\in \mathfrak d\}|\subseteq |H^0(\ma I_{\{q_1^2,\dots,q_8^2,q_9\}/\pr^2}\otimes\ol_{\pr^2}(6))|,$$
it suffices to show that $h^0(\ma I_{\{q_1^2,\dots,q_8^2,q_9\}/\pr^2}\otimes\ol_{\pr^2}(6))=3$.

We apply  \eqref{horace} to $W=\pr^2$, $X=\{q_1^2,\dots,q_8^2,q_9\}$ and $D$ a smooth cubic in $\mathfrak d$, and we get
$$
0 \to \ma I_{\{q_1,\dots,q_8\}/\pr^2}\otimes\ol_{\pr^2}(3)\to\ma I_{\{q_1^2,\dots,q_8^2,q_9\}/\pr^2}\otimes\ol_{\pr^2}(6) \to \ma L:=\ma I_{\{q_1^2,\dots,q_8^2,q_9\}/D}\otimes{\ol_{\pr^2}(6)}_{|D}\to 0,
$$
since $D\cap q_i^2$ is the double point of $D$ with support $q_i$, and ${\rm Res}_D(q_i^2)=q_i$ for all $i\in\{1,\dots,8\}$.
The scheme $D\cap q_i^2$ has lenght $2$ for all $i\in\{1,\dots,8\}$, thus the line bundle $\ma L$ has degree $6\cdot3-8\cdot 2-1=1$, and so $h^0(\ma L)=1$.
Finally, we conclude by $h^0(\ma I_{\{q_1,\dots,q_8\}/\pr^2}\otimes\ol_{\pr^2}(3))=2$ and $h^1(\ma I_{\{q_1,\dots,q_8\}/\pr^2}\otimes\ol_{\pr^2}(3))=0$.
\end{proof}

\begin{proof}[Proof of Proposition \ref{prop:base-locusV}]
Let $q_1,\dots,q_8\in\pr^2$ be the points corresponding to $\Pi_1^{\ma V},\dots,\Pi_8^{\ma V}$ via the Veronese embedding $\nu\colon \pr^2 \hookrightarrow \pr^5$.
That is, $\ma V=\nu(\pr^2)$, $\Pi^{\ma V}_i=T_{\nu(q_i)}\ma V$ for all $i\in\{1,\dots,8\}$, and the blow-up of $\pr^2$ at $\{q_1, \dots,q_8\}$ is a smooth del Pezzo surface.
Moreover, let $q_9\in \pr^2$ be the unique point such that
$$\Bs\left(\mathfrak d:=|H^0(\ma I_{\{q_1,\dots,q_8\}/\pr^2}\otimes\ol_{\pr^2}(3))|\right)=\{q_1, \dots, q_9\},$$
and set $\Pi^{\ma V}_9:=T_{\nu(q_9)}\ma V$.

Note that \emph{i.}\ follows by contruction.
To prove \emph{ii.}, we show that $\Pi^{\ma V}_9\subset Z$ for all $Z \in |H^0(\mathfrak I_{8,3})|$, so that $\Pi_9^{\ma V}\subseteq \Bs(|H^0(\mathfrak I_{8,3})|)$; recall that $\Pi_1^{\ma V}\cup\dots\cup\Pi_8^{\ma V}\subseteq \Bs(|H^0(\mathfrak I_{8,3})|)$ by definition.

Assume by contradiction that there exists $Z \in |H^0(\mathfrak I_{8,3})|$ such that $\Pi^{\ma V}_9\not\subset Z$.
The intersection $Z\cap\Pi^{\ma V}_9$ is a plane cubic $C$ in $\mathfrak d$, under the isomorphism $\theta_9\colon \pr^2 \to \Pi^{\ma V}_9$ (cf.\ Lemma \ref{tangent-planes}), thus $Z$ contains the point $\theta_9(q_9)=\nu(q_9)= (\Pi^{\ma V}_9\cap \ma V)_{red}$.
Below we show that $\Pi_9^{\ma V}\subset T_{\nu(q_9)}Z$, so that $C$ is singular at $\nu(q_9)$; on the other hand, all plane cubics in $\mathfrak d$ are smooth at $\Bs(\mathfrak d)$, and we get a contradiction.

If $\ma V\subset Z$, we clearly have $\Pi_9^{\ma V}\subset T_{\nu(q_9)}Z$.
Otherwise, $\ma V\not\subset Z$ and the intersection $Z\cap \ma V$ is a curve $\Gamma\subset \pr^5$ of degree $12$ through $\{\nu(q_1),\dots,\nu(q_9)\}$ that is singular at $\nu(q_i)$ for all $i\in\{1,\dots,8\}$, since $\Pi_i^{\ma V}=T_{\nu(q_i)}\ma V\subset Z$ for all $i\in\{1,\dots,8\}$.
It follows that $\nu^{-1}(\Gamma)$ is a plane sextic in $|H^0(\ma I_{\{q_1^2,\dots,q_8^2,q_9\}/\pr^2}\otimes\ol_{\pr^2}(6))|$, and so $\nu^{-1}(\Gamma)$ is singular at $q_9$ by Lemma \ref{sextics}.
Again, we conclude that $\Pi_9^{\ma V}\subset T_{\nu(q_9)}Z$.
\end{proof}

As a direct outcome of Proposition \ref{prop:base-locusV} we have the following.
\begin{corollary}\label{co:base-locusV}
Let $Z$ be a smooth cubic fourfold containing planes $\Pi_1^{\ma V},\dots,\Pi_8^{\ma V}$ as in Corollary \ref{existence planes V}, and let $\sigma \colon Z_8 \to Z$ be the blow-up along $\Pi_1^{\ma V}\cup\dots\cup\Pi_8^{\ma V}$.
There exists a ninth plane $\Pi_9^{\ma V}\subset Z$ tangent to $\ma V$ such that the base locus $\Bs(|-K_{Z_8}|)$ contains the strict transform of $\Pi_9^{\ma V}$, that is a smooth del Pezzo surface of degree $1$.
\end{corollary}
\begin{remark}\label{rem:restriction base locus}
It follows from \eqref{eq:restriction base locus} that
$$
\Bs\left(|H^0(\ma I_{\Pi_1\cup\dots\cup\Pi_m/Z}\otimes\ol_{Z}(3))|\right)=\Bs\left(|H^0(\mathfrak I_{m,3})|\right) \ \forall \,m\geq1.
$$
\end{remark}
\begin{proof}[Proof of Corollary \ref{co:base-locusV}]
By Proposition \ref{prop:base-locusV}\emph{ii.}\ and Remark \ref{rem:restriction base locus} we have
$$\Pi_1^{\ma V}\cup\dots\cup\Pi_9^{\ma V}\subseteq\Bs\bigl(|H^0(\ma I_{\Pi_1^{\ma V}\cup\dots\cup\Pi_8^{\ma V}/Z}\otimes\ol_{Z}(3))|\bigr).$$
It follows that the base locus $\Bs(|-K_{Z_8}|)$ contains the strict transform of $\Pi_9^{\ma V}$, that is the blow-up of $\Pi_9^{\ma V}$ at $\{\Pi_1^{\ma V}\cap\Pi_9^{\ma V}, \dots, \Pi_8^{\ma V}\cap\Pi_9^{\ma V}\}$, and the last claim follows from Lemma \ref{tangent-planes}.
\end{proof}

\begin{corollary}\label{abelian fibration}
Let $\Pi_1^{\ma V},\dots,\Pi_9^{\ma V}\subset Z$ be as in Corollary \ref{co:base-locusV}.
Assume that the scheme-theoretic base locus $\Bs(|-K_{Z_8}|)$ is the strict transform of $\Pi_9^{\ma V}$, and let $\sigma \colon Z_9 \to Z$ be the blow-up along $\Pi_1^{\ma V}\cup\dots\cup\Pi_9^{\ma V}$.
Then, $Z_9$ is smooth and there exists an abelian fibration $\Phi\colon Z_9 \to \pr^2$ such that $-K_{Z_9}\simeq\Phi^*\ol_{\pr^2}(1)$.
\end{corollary}
\begin{proof}
We have that $Z_9$ can be realised as the blow-up $Z_9 \to Z_8$ along the scheme-theoretic base locus $\Bs(|-K_{Z_8}|)$ by Corollary \ref{blowup-planes}, so that $Z_9$ is smooth and $-K_{Z_9}$ is globally generated.
We have $h^0(-K_{Z_9})=h^0(-K_{Z_8})=3$ by Proposition \ref{expected-dim} and \eqref{eq:restriction base locus}, thus the linear system $|-K_{Z_9}|$ defines a morphism
$$\Phi\colon Z_9 \to \pr^2.$$
For general $Y_1,Y_2\in|H^0(\ma I_{\Pi_1^{\ma V}\cup\dots\cup\Pi_8^{\ma V}/Z}\otimes\ol_{Z}(3))|$ we have
$$Y_1\cap Y_2=\Pi_1^{\ma V}\cup\dots\cup\Pi_9^{\ma V}\cup U$$
where $U\subset \pr^5$ is a surface of degree $18$.
We deduce that $-K_{Z_9}$ has numerical dimension $2$, hence $\Phi$ is surjective and $-K_{Z_9}\simeq\Phi^*\ol_{\pr^2}(1)$.

The strict transforms $Y_1',Y_2'\subset Z_9$ of $Y_1,Y_2$ are general in $|-K_{Z_9}|$, so the strict transform $F=Y_1'\cap Y_2'$ of $U$ is a general $\Phi$-fibre, thus a smooth surface, and $\sigma_{|F}\colon F\to U$ is a resolution of singularities.
By adjunction we have $K_F\simeq \ol_F$, hence $F$ is either a smooth $K3$ surface or a smooth abelian surface, and $U$ has at most nodal singularities.

Assume that $F$ is $K3$, and let $\ma L$ be the very ample line bundle on $U$ defining the embedding $U\subset \pr^5$. Then, $(\sigma_{|F})^*{\ma L}$ is nef and big and
$${\ma L}^2=((\sigma_{|F})^*{\ma L})^2=8$$
by Riemann-Roch. Thus $U$ has degree $8$ in $\pr^5$, a contradiction.
\end{proof}

\begin{remark}\label{last-added remark}
Note that  we do not assume that $\Pi_1^{\ma V},\dots,\Pi_8^{\ma V}\subset Z$ satisfy the generality condition \textbf{G2} in Corollaries \ref{co:base-locusV} and \ref{abelian fibration}, thus $Z_8$ needs not to be Fano (cf.\ Lemma \ref{condition:3} and Theorem \ref{m=8}).

If $Z_8$ is Fano, then $|-K_{Z_8}|$ has no fixed part by \cite[Theorem 5.2]{Kawa00}, and the strict transform of $\Pi_9^{\ma V}$ is an irreducible component of $\Bs(|-K_{Z_8}|)$.
\end{remark}

\begin{proof}[Proof of Theorem \ref{thm:base-locus}]
We may assume that $\Pi_1,\dots,\Pi_m\subset Z$ are as in Setup \ref{setup} and satisfy the generality conditions $\boldsymbol{(\star)}$, so by Corollary \ref{co:base-locusV} and Remark \ref{last-added remark} it remains to show that $-K_{Z_m}$ is globally generated if and only if $m\leq7$.

\begin{step}\label{restriction-divisor}
$-K_{Z_8}$ is not globally generated.
\end{step}
\begin{proof}
We have $h^0(-K_{Z_8})=3$ by Remark \ref{anticanonical sections}; see also Table \ref{final table}.
If $-K_{Z_8}$ is globally generated, then the linear system $|-K_{Z_8}|$ defines a morphism $Z_8 \to \pr^2$ and $(-K_{Z_8})^4=0$, a contradiction.
\end{proof}

Since $\Bs(|-K_{Z_{m-1}}|)\subseteq\sigma_m\bigl(\Bs(|-K_{Z_m}|)\bigr)$ for all $m\geq1$, where $\sigma_m\colon Z_m\to Z_{m-1}$ is the blow-up along the strict transform of $\Pi_m$, it suffices to show that $-K_{Z_7}$ is globally generated.

Let $S_7$ be a del Pezzo surface of degree $2$, that is the blow-up $g\colon S_7\to \pr^2$ along $7$ points in general position.
Recall the isomorphism $\alpha\colon \ma N^1(S_7)\to\ma N^1(Z_7)$ from Remark \ref{alpha}, and the definition of ${\rm Cris}(Z_7)$ and $\Exc(Z_7)$ from Remark - Definition \ref{cris_X}.

\begin{step}\label{vector-bundle}
$-K_{Z_7}$ is globally generated.
\end{step}
\begin{proof}
We deduce from \eqref{anticanonical} that
$-K_{Z_7}\simeq E_i+H_{3,i}$ for all $i\in \{1, \dots, 7\}$,
where
$$H_{3,i}\simeq 3H-2E_i-(E_1+ \dots + E_7-E_i)$$
is the strict transform of $H_{3,\,\mathfrak p_{(6,1)}}$ from Corollary \ref{Eff(X)}\emph{iii.}\ with $i=b_1$, and $H_{3,i}\in\Exc(Z_7)$.
Therefore,
$$
\Bs(|-K_{Z_7}|)\subseteq \bigcap_{i=1}^7(E_i\cup H_{3,i}) =:\star,
$$
where $\star$ is the union of all intersections $D_1 \cap \dots\cap D_7$ such that $D_i$ is either $E_i$ or $H_{3,i}$ for all $i\in \{1, \dots, 7\}$.
Note that every $\{D_1, \dots, D_7\}$ contains al least three divisors among the $E_i$'s or among the $H_{3,i}$'s.

Let $l_{3,i}:=\alpha^{-1}(H_{3,i})$ be the corresponding $(-1)$-curve of $S_7$.
We have $l_{3,i}\cdot l_{3,j}=0$ for all $i\neq j$, hence $\{l_{3,1}, \dots, l_{3,7}\}$ are the exceptional curves of a blow-up $g'\colon S_7\to \pr^2$.
It follows from Theorem \ref{m=8}\emph{ii.}\ that $\{H_{3,1}, \dots, H_{3,7}\}$ are the exceptional divisors of $\sigma_\zeta \colon Z_7 \to Z_\zeta$ for some $\zeta \in {\rm Cris}(Z_7)$, and so $H_{3,i}\cap H_{3,j}\cap H_{3,k}=\emptyset$ for all $1\leq i< j <k \leq 7$ (cf.\ Setup \ref{setup}).
Similarly, we have $E_i\cap E_j\cap E_k=\emptyset$ for all $1\leq i< j <k \leq 7$.
We conclude that $\star=\emptyset$, and the claim follows.
\end{proof}
\end{proof}

\section{Invariants and final table}\label{final}
In this section we compute and collect the numerical invariants of the Fano fourfold $Z_m$ from Theorem \ref{main:fano} and comment on its rationality.

\subsection{Numerical invariants}\label{num.inv.}
Some of the invariants in Table \ref{final table} already appear in \cite[Table 11.3]{Casa25}, that is $K_{Z_m}^4$, $h^{2,2}(Z_m)$, $h^{1,3}(Z_m)$, $b_3(Z_m)=2h^{1,2}(Z_m)$ and $\chi(Z_m-K_{Z_m})$.
These can be computed, for instance, via the sequence of smooth blow-ups
$$Z_m \rightarrow Z_{m-1} \to \dots \to Z_1 \rightarrow Z$$
and the formulas in \cite[Lemma 3.2]{CR} and \cite[Lemma 3.1]{Secci23}.

To compute $K_X^2\cdot c_2(X)$ on a smooth Fano fourfold $X$, recall that the Kodaira vanishing and Hirzebruch-Riemann-Roch theorems give
$$h^0(X, -K_X)=\chi(X, -K_X)=1+\frac{1}{12}\bigl(2K_X^4+K_{X}^2\cdot c_2(X)\bigr).$$

The cohomology of the tangent bundle $T_X$ can instead be deduced from the above invariants via \cite[Lemma 6.25]{CCF}, i.e.\
$$\chi(X, T_X)=27-5h^0(X, -K_X)+K_X^4+3\rho(X)-h^{1,2}(X)-h^{2,2}(X)+3h^{1,3}(X).$$
We have $h^0(Z,  T_{Z})=0$ by the Matsumura-Monsky Theorem \cite{MM63}, thus $h^0(Z_m,  T_{Z_m})=0$ for all $m\geq0$ since $h^0(Z_{m+1},  T_{Z_{m+1}})\leq h^0(Z_{m}, T_{Z_{m}})$.
Moreover, $h^i(Z_m, T_{Z_m})=0$ for all $i\geq 2$ by the Nakano vanishing \cite[Theorem 4.2.3]{LazI}, so $\chi(Z_m,T_{Z_m})=-h^1(Z_m, T_{Z_m})$.

\begin{table}[h] \caption{\small Invariants of the blow-up $Z_m$ of the cubic fourfold along $m$ planes.}\label{final table}
\begin{adjustbox}{center}
\begin{tabular}{||c|c|c|c|c|c|c|c|c|c||}
\hline\hline
\rule{0pt}{2.5ex} $m$  & $\rho(Z_m)$ & $K_{Z_m}^4$ & $K_{Z_m}^2\cdot c_2(Z_m)$ & $h^{2,2}(Z_m)$ & $h^{1,3}(Z_m)$ & $h^{1,2}(Z_m)$ & $h^0(-K_{Z_m})$ & $h^0( T_{Z_m})$ & $h^1( T_{Z_m})$ \\
\hline\hline
\rule{0pt}{2.5ex} 0 & 1 & 243 & 162 & 21 & 1 & 0 & 55 & 0 & 20 \\
\hline
\rule{0pt}{2.5ex} 1 & 2 & 192 & 144 & 22 & 1 & 0 & 45 & 0 & 19 \\
\hline
\rule{0pt}{2.5ex} 2 & 3 & 147 & 126 & 24 & 1 & 0 & 36 & 0 & 18 \\
\hline
\rule{0pt}{2.5ex} 3 & 4 & 108 & 108 & 27 & 1 & 0 & 28 & 0 & 17 \\
\hline
\rule{0pt}{2.5ex} 4 & 5 & 75 & 90 & 31 & 1 & 0 & 21 & 0 & 16 \\
\hline
\rule{0pt}{2.5ex} 5 & 6 & 48 & 72 & 36 & 1 & 0 & 15 & 0 & 15 \\
\hline
\rule{0pt}{2.5ex} 6 & 7 & 27 & 54 & 42 & 1 & 0 & 10 & 0 & 14 \\
\hline
\rule{0pt}{2.5ex} 7 & 8 & 12 & 36 & 49 & 1 & 0 & 6 & 0 & 13 \\
\hline
\rule{0pt}{2.5ex} 8 & 9 & 3 & 18 & 57 & 1 & 0 & 3 & 0 & 12 \\
\hline\hline
\end{tabular}
\end{adjustbox}
\end{table}

We conclude with a remark on the Lefschetz defect $\delta$, a numerical invariant introduced in \cite{Casa12} that has proved fundamental to the study of higher-dimensional Fano manifolds, particularly in the classification of Fano fourfolds of large Picard rank\footnote{See, for instance, \cite{CasaSurvey}, \cite{Casa12,CR,CRS} for $\delta\geq 3$, \cite{Casa14, Secci23, Pier25,Casa26} for $\delta=2$, \cite{Casa24,CS24,Casa25} for large Picard rank, and \cite{Verni26} for its behaviour in families.}.
Given a smooth Fano variety $X$, the Lefschetz defect is defined as
$$\delta(X):=\max\bigl\{\codim\ma N_1(D,X)\,|\,D\text{ is a prime divisor in }X\bigr\},$$
where $\ma N_1(Y,X)$ is the image of the push-forward $\iota_*\colon \ma N_1(Y) \to \ma N_1(X)$ induced by the inclusion $\iota \colon Y \hookrightarrow X$ of a subvariety $Y$.
In our setting we have
$$
0\leq\delta(T),\,\delta(Z_1) \leq \delta(Z_2)\leq \dots \leq \delta(Z_8)\leq 1
$$
by \cite[Remark 3.3.18]{Casa12} and \cite[Theorem 1.8]{CS24}, where $T$ is the Verra fourfold from Remark \ref{H_12}.
Note that $\codim \ma N_1(E,Z_m)=0$ for all $E\in \Exc(Z_m)$ and all $m\geq 1$.
Thus, either $\delta(Z_m)=0$, or $\delta(Z_m)=1$ and any prime divisor $D \subset Z_m$ satisfying $\codim \ma N_1(D,Z_m)=1$ is the strict transform of an irreducible hypersurface section of $Z\subset \pr^5$.

\subsection{Rationality}
Although we did not investigate the (ir)rationality of $Z_m$ for any $m\geq 1$, we conclude the paper with a few comments on the subject.

By construction, $Z_m$ is birational to a smooth cubic fourfold $Z\subset \pr^5$ and, for $m\geq 2$, also to a smooth Verra fourfold $T \subset \pr^9$ (see Remark \ref{H_12}).

Irrationality of the very general cubic fourfold was proved in \cite[Theorem 6.8]{KKPY}, but the problem is still open for the very general Verra fourfold.
Nonetheless, equivariant irrationality for the very general \emph{symmetric} Verra fourfold was proved in \cite[Theorem 6.4]{Fay26} using similar techniques from \cite{CKK,KKPY}, while in \cite[Theorem 1]{BGMP} the authors show that a Verra fourfold can never be birational to a very general cubic fourfold.

In our setting, $Z_m$ is the blow-up of a smooth cubic fourfold $Z$ containing $m\geq 1$ planes as in Setup \ref{setup} and so $Z$ is not general in the complete linear system $|\ol_{\pr^5}(3)|$.

Lastly, we remark that a necessary condition for the rationality of smooth cubic fourfolds was proved in \cite[Theorem 56]{Guere26}.

\appendix\section{}
Notation as in Section \ref{sec:existence}.
The following computations were performed using Macaulay2 \cite{M2}.
\begin{lemma}\label{appendix}
Let $\Pi_1^{\ma V},\dots,\Pi_8^{\ma V}$ be planes as in Corollary \ref{existence planes V}, and let $B$ be the scheme-theoretic base locus $Bs(|H^0(\pr^5,\mathfrak I_{8,3})|)$.
Then, the general $Z\in |H^0(\pr^5,\mathfrak I_{8,3})|$ is smooth.
\end{lemma}
\begin{proof}
\begin{step}
We generate general points $q_1, \dots, q_8$ in $\pr^2_{\Q}$.
\end{step}
\begin{verbatim}
randomPoint = () -> (p := for i to 2 list random(-10,10)/random(1,10);
if p == {0,0,0} then randomPoint() else p);

pts = for i from 1 to 8 list randomPoint();
\end{verbatim}

\begin{step}
We map each $q_i$ to $\pr^5_{\Q}$ via the $2$-Veronese embedding $\nu\colon \pr^2_{\Q} \hookrightarrow \pr^5_{\Q}$, and we compute the ideal of the tangent planes $\Pi^{\ma V}_i:=T_{\nu(q_i)}\ma V$ for all $i\in \{1, \dots, 8\}$.
\end{step}
\begin{verbatim}
R = QQ[x0,x1,x2,x3,x4,x5];

v = p -> {p#0^2, p#1^2, p#2^2, 2*p#0*p#1, 2*p#0*p#2, 2*p#1*p#2};

Veronese = ideal(4*x0*x1 - x3^2, 2*x0*x5 - x3*x4, 4*x0*x2 - x4^2,
x3*x5 - 2*x1*x4, 2*x2*x3 - x4*x5, 4*x1*x2 - x5^2);

JV = jacobian Veronese;

TangentPlane = p -> trim ideal(matrix{gens R} * substitute(JV, matrix{p}));

Planes = apply(pts, p -> TangentPlane v p);
\end{verbatim}

\begin{step}
We obtain the linear system $|H^0(\pr^5_{\Q},\mathfrak I_{8,3})|$ of cubic hypersurfaces in $\pr^5_\Q$ containing $\Pi_1^{\ma V}\cup\dots\cup\Pi_8^{\ma V}$ by intersecting the ideals of the $\Pi_i^{\ma V}$'s and selecting the generators of degree $3$.
\end{step}
\begin{verbatim}
I = intersect Planes;

Cubics = ideal select(flatten entries gens I, f -> (degree f)_0 == 3);
\end{verbatim}
In fact, there is no hyperplane containing $3$ of the $\Pi_i^{\ma V}$'s by Corollary \ref{existence planes V} and there is no quadric containing $6$ of the $\Pi_i^{\ma V}$'s by Lemma \ref{I_6(2)}, so the generators of $I$ have degree $\geq 3$ and it simply suffices to select the generators of degree $3$.

\begin{step}\label{smoothness}
We show that a general $Z\in|H^0(\pr^5,\mathfrak I_{8,3})|$ is smooth.
\end{step}
In order to show that a general $Z\in|H^0(\pr^5_{\Q},\mathfrak I_{8,3})|$ is smooth, we verify that the saturation of its Jacobian ideal with respect to the irrelevant ideal is the whole coordinate ring.
\begin{verbatim}
m = ideal gens R;

IZ = ideal random(3, Cubics);

JZ = IZ + ideal jacobian IZ;

Jsat = saturate(JZ, m);

Jsat == R  -- true means Z is smooth
\end{verbatim}
We deduce that $Z$ is also smooth over $\C$ since smoothness is invariant under base change \cite[III, Proposition 10.1(b)]{Hart}.
\end{proof}

\providecommand{\noop}[1]{}
\providecommand{\bysame}{\leavevmode\hbox to3em{\hrulefill}\thinspace}
\providecommand{\MR}{\relax\ifhmode\unskip\space\fi MR }
\providecommand{\MRhref}[2]{%
  \href{http://www.ams.org/mathscinet-getitem?mr=#1}{#2}
}
\providecommand{\href}[2]{#2}

\end{document}